%% file: ddd.tex
\documentclass[]{scrartcl}
\usepackage[]{custom-article}
\usepackage{custom-bib}
\usepackage[dvipsnames]{xcolor}
\usepackage{siunitx}
\usepackage{booktabs}
\usepackage{pgf}
\usepackage{orcidlink}

\usepackage{todonotes}
\newcommand{\Alex}[2][]{
	\ifthenelse{\isempty{#1}}{
		\todo[color=pink!70!white!30,author=\textbf{Alex},inline]{ #2}
	}{
		\todo[color=pink!70!white!30,author=\textbf{Alex},inline,caption={#1}]{ #2}
	}
}

\newcommand{\consolidationFormat}[1]{}

\newcommand{\snd}{\complexitylangformat{SND}\xspace}
\newcommand{\rsnd}{\complexitylangformat{R-SND}\xspace}
\newcommand{\orp}{\complexitylangformat{ORP}\xspace}
\newcommand{\rorp}{\complexitylangformat{R-ORP}\xspace}
\newcommand{\solsym}{S}
\newcommand{\pathsym}{\calP}
\newcommand{\spathsym}{P}
\newcommand{\walksym}{W}
\newcommand{\conssym}{\calC}
\newcommand{\edt}[2][k]{\underline{t}_{#1}({#2})}
\newcommand{\ldt}[2][k]{\overline{t}_{#1}({#2})}
\newcommand{\vdn}[2]{({#1},{#2})}
\newcommand{\rounddown}[2]{\lfloor#1\rfloor_{\calT_{#2}}}
\newcommand{\rat}[2][\walksym]{\pi_{#1}(#2)}
\newcommand{\ratp}[2][\walksym]{\pi'_{#1}(#2)}
\newcommand{\ratw}[1]{\pi_{#1}}
\newcommand{\rdt}[2][\walksym]{t_{#1}(#2)}
\newcommand{\tlpsymb}{\mathcal{W}}

\newcommand{\labels}{L}
\newcommand{\starts}{\calS}
\newcommand{\boundaries}{\calB}

\newcommand{\exgraph}{D_\tlpsymb}
\newcommand{\candgraph}{\exgraph^\uparrow}

\newcommand{\bigO}{\mathcal{O}}
\newcommand{\orpsol}{\Phi}

\LinesNumbered

\DeclareMathOperator{\delay}{\Delta_S}
\DeclareMathOperator{\exc}{\Delta_L}
\DeclareMathOperator{\capacity}{cap}
\DeclareMathOperator{\requirement}{req}

\makeatletter
\renewcommand{\maketitle}{\par\begingroup
	\setlength{\parindent}{0pt}\setlength{\parskip}{0pt}\raggedright
	{\normalfont\Large\scshape\@title\par}\vspace{.8em}{\normalfont\normalsize\@author
		\hfill{\footnotesize\color{black!60}}\par}\endgroup\par
}
\makeatother

\title{Optimal Refinement in Dynamic Discretisation Discovery
	for Continuous-Time Service Network Design}
\author{Helber, Alexander\,\orcidlink{0009-0006-3482-4632}, Bachtler, Oliver\,\orcidlink{0000-0001-7942-0750}\\
	{\small Chair of Operations Research, RWTH Aachen University, Germany\\
	\href{mailto:helber@or.rwth-aachen.de}{helber@or.rwth-aachen.de}\enspace\textbar\enspace
	\href{mailto:bachtler@or.rwth-aachen.de}{bachtler@or.rwth-aachen.de}}\\[.4em]
	Wüllner, Tom\,\orcidlink{0009-0000-1782-1372}\\
	{\small RWTH Aachen University, Germany\\
	\href{mailto:tom.wuellner@rwth-aachen.de}{tom.wuellner@rwth-aachen.de}}}

\begin{document}
\maketitle
\thispagestyle{plain}

\input{sections/abstract.tex}
\begingroup
\setlength{\parskip}{0pt}
\RedeclareSectionCommand[beforeskip=.4\baselineskip]{paragraph}
\par
\endgroup

\input{sections/introduction.tex}

\input{sections/snd.tex}
\input{sections/ddd-algo.tex}
\input{sections/refinement.tex}
\input{sections/variants.tex}
\input{sections/computational-study.tex}

\input{sections/conclusion.tex}

\printbibliography
\clearpage
\appendix
\addpart{Appendix}
\input{sections/relaxation-mip.tex}

\end{document}

%% file: sections/abstract.tex
\paragraph{Abstract.}
Dynamic Discretisation Discovery (DDD) is a framework for solving a problem by iteratively solving and refining a relaxed version of the problem.
DDD is the basis for state-of-the-art algorithms for various problems, including the continuous-time service network design problem (CTSNDP), which we study here.
In the refinement step, such algorithms must identify \emph{conflicts} in the relaxation solution that make it infeasible for the original problem and then apply \emph{fixes} to the relaxation to prevent them from recurring.
We study the \emph{optimal refinement problem} of selecting from a set of fixes a subset that fixes all conflicts and yields a relaxation of minimum size, with the aim of making it easier to solve.

We demonstrate that while the number of conflicts may be exponential in the instance parameters, we can represent them more compactly, using only pseudo-polynomial space and time.
Based on our more compact representation, we develop exact methods for solving the refinement problem.
The best method allows us to find small relaxations in negligible amounts of time, more than 100 times faster than a natural integer programming formulation for some instance classes.
Using this refinement approach in a DDD algorithm for the CTSNDP, we observe reductions in model sizes and running times of about \qtyrange{10}{20}{\percent} compared to existing approaches.

\noindent\hrulefill

%% file: sections/introduction.tex
\section{Introduction}
Service network design problems arise in the planning of transport systems, especially for consolidation-based carriers \cite{crainic_network_2021}.
These carriers transport shipments that are usually much smaller than the capacity of the vehicles used, so economical operations rely on consolidating shipments.
Therefore, the routing of commodities and the dispatches of vehicles transporting them must be determined jointly.
We study the continuous-time service network design problem~\cite{boland_continuous-time_2017}, in which dispatch times must also be determined, while ensuring that all commodities arrive by their respective deadlines.

Current state-of-the-art approaches~\cite{helber_arc-based_2026,marshall_interval-based_2021,shu_new_2025} for this problem are all based on the dynamic discretisation discovery (DDD) framework first introduced by~\textcite{boland_continuous-time_2017}.
So far, DDD has been mostly applied to problems that can be modelled on time-expanded networks. 
We note that the DDD concept is more general and can be applied to other expanded graphs. For example, \textcite{Riedler2019} describe its application to the rooted distance-constrained minimum spanning tree problem.
In the remainder of the introduction, we will focus on time-based problems, but the general ideas are relevant for any problem amenable to DDD.

The DDD framework is based upon the idea of constructing a relaxed version of the problem to be solved, also called a \emph{relaxation}.
This relaxed problem is formulated on a time-expanded network based on a given discretisation of time (for the nodes/arcs of the network that is being expanded).
If the discretisation is sufficiently sparse, the relaxed problem may be solved much faster than the original problem, yet its optimal solutions provide good dual bounds on the value of the original problem.

We say a solution that is feasible for the relaxed problem is \emph{representable} with respect to the chosen discretisation \cite{marshall_interval-based_2021} and call feasible solutions to the original problem \emph{implementable}.
In the framework, an optimal solution to the relaxation is computed and, if it is implementable (or can be converted into an implementable solution of the same value), we can terminate with a provably optimal solution.
Otherwise, we refine the discretisation (and thus the relaxation) in such a way that the non-implementable solution also becomes non-representable and cannot recur, guaranteeing eventual convergence.
This process is also outlined in~\cref{fig:ddd-flowchart}.

\begin{figure}[tbp]
    \centering
    \begin{tikzpicture}[
        node distance=4mm and 8mm,
        process/.style={draw, rectangle, align=center, text width=45mm, minimum height=10mm},
        flow/.style={->, >=stealth, thick}
    ]
        \node[process] (initialise) {Initialise relaxation};
        \node[process, below=of initialise] (solve) {Find optimal \mbox{solution~$\solsym$} for the relaxation};
        \node[process, below=of solve] (check) {Is $\solsym$ implementable?};
        \node[process, right=of check, text width=10mm] (terminate) {Stop};
        \node[process, below=of check] (refine) {Refine relaxation to make\\$\solsym$ non-representable};

        \draw[flow] (initialise) -- (solve);
        \draw[flow] (solve) -- (check);
        \draw[flow] (check) -- node[above] {Yes} (terminate);
        \draw[flow] (check) -- node[right] {No} (refine);
        \draw[flow] (refine.west) -- ++(-8mm,0) |- (solve.west);
    \end{tikzpicture}
    \caption{Overview of the dynamic discretisation discovery approach.}
    \label{fig:ddd-flowchart}
\end{figure}
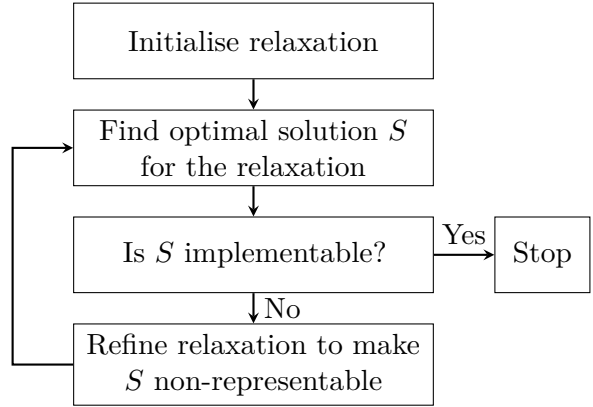

In this work, we mostly concern ourselves with the refinement step, thus it is interesting to reflect on existing approaches. 
When designing a DDD algorithm, we require some concept of \emph{conflicts} that prevent a representable solution from being implementable \cite{boland_perspectives_2019,marshall_interval-based_2021}.
Then, we would like to have one or more \emph{fixes} available for each conflict; if a fix is applied to the discretisation, the solution becomes non-representable.
A fix usually consists of adding one or more time points to the discretisation.
It may be possible to describe multiple types of conflicts, and for each conflict type multiple types of fixes. 

While it suffices to fix a single conflict in each iteration to guarantee convergence, it has been observed that fixing many conflicts at once can decrease the number of iterations required.
But, a fix may increase the size of the relaxed problem, making it harder to solve.
So to design an efficient DDD algorithm, one has to decide which (types of) conflicts to resolve and with which fixes, balancing relaxation difficulty against the number of iterations.
This trade-off has been observed in many studies where different refinement strategies are compared \cite{boland_perspectives_2019,he_exact_2022,he_dynamic_2022,marshall_interval-based_2021,Riedler2019,scherr_dynamic_2020}.
It is also possible to fix some conflicts by adding valid inequalities to the relaxed problem instead of changing the discretisation \cite{boland_perspectives_2019,hewitt_enhanced_2019}, but this leads to essentially the same trade-off.

For various problem classes, \textcite{marshall_interval-based_2021} and \textcite{Riedler2019} find that it generally pays off to fix (almost all) observed problems, but using fixes that increase discretisation size as little as possible. 
To this end, \textcite{marshall_interval-based_2021} also identify conflicts that lie \enquote{upstream} of other conflicts and prefer to fix these first, in the hope that this also fixes the downstream conflicts. 

The hope that an observed conflict may not recur if other conflicts are fixed is also reflected in the idea of removing time points from the discretisation \cite{hewitt_enhanced_2019,Riedler2019,scherr_dynamic_2020}.
We also observed in preliminary experiments that, once DDD converges, quite a substantial number (often more than \SI{30}{\percent}) of time points can be removed from the discretisation before the bound quality changes. 
Yet it is unclear which time points are safe to remove; the removal of a necessary time point may cause another iteration of the algorithm.
In fact, \textcite{van_dyk_routing_2024} proves that a naive strategy that simply removes time points that are not in the support of the solution can lead to cycling, preventing the algorithm from converging.

To summarise, we have observed the following:
there is substantial freedom in which conflicts to fix (and how), the chosen fixes influence the difficulty of solving the relaxed problem and the number of DDD iterations, many added time points may not be required in later iterations, but removing time points naively can prevent algorithm convergence.
These observations motivate us to ask the following question:
\begin{center}
    Can we find a minimum discretisation that fixes all observed conflicts?
\end{center}
Finding such a discretisation in each iteration would directly address the last concern: the new discretisation would not necessarily contain the time points of the previous one, effectively allowing us to remove time points. 
But since all previously observed conflicts remain fixed, no cycling can occur.
Since the discretisation is small (with respect to some so far undefined size measure), we may hope for relaxations that are easier to solve.

This approach would also make full use of situations where many different fixes are available for each conflict, selecting a convenient subset of them.
Note that for the scope of this study, we set aside the question of which conflicts to fix and commit to fixing all of them. 
We leave the selection of a good subset of conflicts to fix (and keep fixed) to future work.

To the best of our knowledge, the literature that explicitly optimises the set of fixes to be applied is limited to the following works.

\textcite{Lieshout26} solve a vehicle scheduling problem with trip shifting using DDD. 
In their approach, a (representable) solution consists of a flow over a time-expanded network that is then decomposed into duties for each vehicle.
If a duty cannot be implemented, the relaxation is refined to ensure that it cannot occur again. 
Since they know how many time points each duty would add to the discretisation, they pose an auxiliary problem that selects a duty decomposition leading to the fewest added time points.
We note that the additional time points they add to prohibit a single non-implementable duty are sufficient, but not necessary to achieve this goal. 
In other words, there may be more fixes available than the ones they utilise.
They also do not account for double-counting of time points that might be added by multiple duties.
In their computational experiments, they observe little change in running time from finding such smaller discretisations.
The approach is outperformed by a method that simply refines all non-implementable duties possible in the relaxation solution, leading to marginally larger networks but fewer iterations and shorter running times.

For the service network design problem studied in this work, \textcite{shu_new_2025} proposed a conflict concept that subsumes the concepts introduced by \textcite{marshall_interval-based_2021}. 
Notably, they introduce a sufficient characterisation of a discretisation that prevents a conflict from recurring.
This characterisation implies that one conflict may permit multiple fixes.
Based on this observation, the authors introduce a greedy heuristic to construct a discretisation that fixes all observed conflicts.
In this work, we introduce a closely related but simpler sufficient criterion for such fixes, and use it to develop various exact methods for solving the \emph{optimal refinement problem} of finding a minimum discretisation that selects one of these fixes for each conflict.
In our short paper~\textcite{Wullner2026-xl} (which the current work extends) we already demonstrated that this problem is NP-hard and that being able to solve it efficiently can be beneficial for the overall DDD algorithm.

We note that shortly before finishing this work, we became aware of the new version \cite{shu_new_2026} of the previously mentioned preprint \cite{shu_new_2025}.
In the updated version, the authors also explicitly pose the optimal refinement problem, using a slightly stronger characterisation than ours.
They further propose a matheuristic to solve the problem. 
We discuss the similarities and differences between the approaches where appropriate throughout the paper.

Our contributions in this paper are:
\begin{itemize}
    \item We provide a focused exposition of the main concepts for refinement introduced by~\textcite{marshall_interval-based_2021} and~\textcite{shu_new_2025,shu_new_2026}, with small extensions to theoretical results. \item We describe situations in which the approach of~\textcite{shu_new_2026} leads to a conflict set whose size is exponential in some instance parameters.
    \item We develop a more compact graph-based representation of the conflict set, whose size is pseudo-polynomial in the instance parameters and much smaller in practice.
    \item This representation allows us to restate the optimal refinement problem and develop a mixed-integer programming formulation as well as a DDD approach to solve it.
    \item We conduct computational experiments to demonstrate the performance of our methods and investigate the impact this refinement approach has on the overall DDD algorithm.
\end{itemize}

%% file: sections/snd.tex
\section{The Continuous-Time Service Network Design Problem}

Before we introduce the continuous-time service network design problem, we want to briefly collect typical notation used throughout the paper.
Here, all our intervals contain only integer values.
Thus, when we write $[m,n]$, for $m$, $n\in\NN$ with $m\leq n$, we refer to the set $\Set{m,\ldots,n}$.
When $m=0$, we also just write $[n]$ instead of $[0,n]$ and we also use this notation for open and half-open intervals such as $(m,n] = \Set{m+1,\ldots,n}$ and $(n) = \Set{1,\ldots,n-1}$.

We write $\V{D}$ and $\A{D}$ for the node and arc sets of a \emph{directed graph} $D$.
For a node $v$, we denote its predecessors and successors $\Nin{v}, \Nout{v}$ and its in- and outgoing arcs $\Ain{v},\Aout{v}$.
A \emph{walk} $\walksym$ in $D$ is a sequence $v_0a_1v_1\ldots a_nv_n$ of nodes and arcs such that $a_i=v_{i-1}v_i$.
Depending on the context, we just denote $\walksym$ by $a_1\ldots a_n$ or $v_0\ldots v_n$.
We also write $v_i\walksym v_j$ for the sub-walk $v_i\ldots v_j$ of $\walksym$, where $v_i\walksym = v_i\walksym v_n$ and $\walksym v_j = v_0\walksym v_j$.
If $\walksym$ repeats no nodes, we call it a \emph{path}.
We also use the typical notation $\tau(\walksym) = \sum_{i=1}^n \tau(a_i)$ for an arc length $\tau\colon\A{D}\to\ZZ_{>0}$.
Additionally, we write $\dist[\tau]{v,v'}$ for the length of a shortest path from $v$ to $v'$ with respect to $\tau$.

Now let us define the continuous-time service network design problem (\snd).
We are given a network $D$ (a directed graph) with nodes $\V{D}$ and arcs $\A{D}$, through which a set of commodities $\calK$ should be transported.
Each commodity $k$ is associated with a release time $r_k \in \NN$ and a deadline $\ell_k \in \NN$, and each arc $a$ with a travel time $\tau(a)\in \NN_{>0}$.
For each commodity, we need a path $\spathsym_k$ in $D$ from its origin node $o_k$ to its destination node $d_k$.
A small example can be seen in \cref{fig:snd-illustration}, where we specify the release time $r_k$ at the origin $o_k$ and deadline $\ell_k$ at the destination node $d_k$ of commodity $k$.
Note that the paths $\spathsym_k$ are unique in this example.

\begin{figure}[tbp]
	\centering
	\begin{tikzpicture}
		\node[draw,label={[align=left]left:$r_{k_1}=1$\\$r_{k_2}=0$}] (v1) {$v_1$};
		\node[draw, right=3 of v1,label={[align=left]above:$\ell_{k_1}=3$\\$r_{k_3}=1$}] (v2) {$v_2$};	
		\node[draw, right=3 of v2,label={[align=left]right:$\ell_{k_2}=8$\\$\ell_{k_3}=3$}] (v3) {$v_3$};
		\draw[->] (v1) to node[below]{$\Set{k_1,k_2}$} node[above]{$\tau(v_1v_2)=2$}(v2);
		\draw[->] (v2) to node[below]{$\Set{k_2,k_3}$} node[above]{$\tau(v_2v_3)=2$}(v3);
	\end{tikzpicture}
	
	\caption{Example of an \snd instance with three commodities $\calK=\Set{k_1,k_2,k_3}$.}
	\label{fig:snd-illustration}
\end{figure}
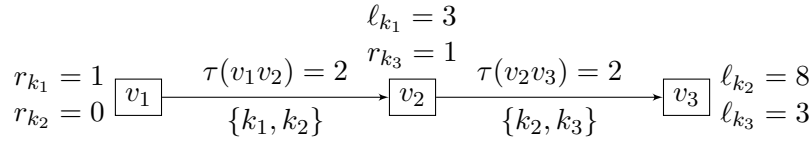

Furthermore, we want to decide for every arc of $\spathsym_k$ with which other commodities $k$ is consolidated on that arc, that is, which commodities traverse the arc together.
If we want to say that the commodities $\kappa \subseteq \calK$ are consolidated on arc $a$, we combine them in a tuple $c = (a,\kappa)$ and call $c$ a \emph{consolidation}.
In the example, we have tried to consolidate everything that could potentially be consolidated, leading to the two consolidations $c_1=(v_1v_2,\Set{k_1,k_2})$ and $c_2=(v_2v_3,\Set{k_2,k_3})$.

Hence, we are effectively looking for a set of consolidations $\conssym \subseteq \A{D} \times 2^\calK$ such that each commodity appears in exactly one consolidation for each arc in its path, and does not appear in any other consolidations. 
A \emph{solution} $\solsym = (\pathsym,\conssym)$ now consists of paths $\pathsym = (\spathsym_k)_{k\in\calK}$ for each commodity and corresponding consolidations~$\conssym$.

Next, we specify which solutions are feasible.
We now require consistent dispatch times for each consolidation, such that all commodities are transported within their respective time windows.
Denote the unique consolidation of the arc $a \in \spathsym_k$ containing $k$ by $c(a,k)$.
Moreover, for a consolidation $c = (a,\kappa)$ we write $a(c) = a$ and $\kappa(c) = \kappa$.
We say the solution \(\solsym\) is \emph{feasible} for SND or \emph{implementable} if dispatch times $t \colon \conssym \to\NN$ exist such that, for all $k\in \calK$ with $\spathsym_k = a_1\ldots a_n$ and $c_i = c(a_i,k)$ for $i\in(n]$,
\begin{enumprop}
	\item $t(c_1) \geq r_k$, 
	\item $t(c_{i+1}) \geq t({c_i}) + \tau(a_i)$ for all $i\in (n)$, and
	\item $t(c_n) + \tau(a_n) \leq \ell_k$.
\end{enumprop}
The first condition ensures that a commodity is not dispatched before it is released, the second that it cannot be dispatched from an intermediate node before arriving there, and the third that it arrives by its deadline.
Here we see that our specified solution is not implementable, since consolidating $k_2$ and $k_3$ on the second arc requires $t(c_2)\geq r_{k_2}+\tau(v_1v_2) = 2$ but $t(c_2)+\tau(v_2v_3)\leq \ell_{k_3}=3$, which is impossible.
By replacing $c_2$ by the singleton \enquote{consolidations} $(v_2v_3,\Set{k_2})$ and $(v_2v_3,\Set{k_3})$, we can obtain a feasible solution: dispatch $c_1$ at time $1$, then $k_2$ on the second arc at time $3$ and $k_3$ at time $1$.

The objective is to minimise the total transportation cost, which consists of two parts. First, a commodity $k$ incurs a flow cost of $c_k(a)$ for traversing arc $a$. 
Second, transporting the commodities in a consolidation $(a,\kappa)$ requires a certain number of vehicles.
Each commodity $k$ requires a quantity $q_k$ to be transported, and for each arc $a$ we are given the capacity $u(a)$ of available vehicles and cost $f(a)$ for dispatching one vehicle on this arc. 
Overall, solving the continuous-time service network design problem (\snd) consists of finding a feasible solution $\solsym$ minimising
\begin{displaymath}
	c(\solsym) = \sum_{k\in\calK}\sum_{a\in \spathsym_k}c_k(a) + \sum_{(a,\kappa)\in\conssym} f(a)\left\lceil \frac{\sum_{k\in\kappa}q_k}{u(a)} \right\rceil.
\end{displaymath}

For the remainder of this paper, we assume that the input data for the problem are given and refrain from stating it in all of our statements.
The following derived instance parameters will also come in handy later:
\begin{definition}[Earliest and latest dispatch times]
	For $k\in \calK$ and $v\in \V{D}$, the \emph{earliest dispatch time} of commodity $k$ at node $v$ is $r_k+\dist[\tau]{o_k,v}$ and is denoted by $\edt{v}$. The \emph{latest dispatch time} of $k$ at $v$, $\ldt{v}$, is $\ell_k-\dist[\tau]{v,d_k}$.
	
	We extend the notation to arcs $a=vv'$ by setting $\edt{a} = \edt{v}$ and $\ldt{a}=\ldt{v'} - \tau(a)$.
	Thus, the interval in which a commodity $k$ can be dispatched on an arc $a$ is specified by $I_k(a) = [\edt{a},\ldt{a}]$.
\end{definition}
We assume that the network is sufficiently connected for these values to be defined.
Additionally, we assume without loss of generality that $\min_{k \in \calK} r_k = 0$ and denote the time horizon $\max_{k \in \calK} \ell_k+1$ by $H$.

\begin{observation}
	\label{intervals-implementable-dt}
	Let $\solsym=(\pathsym,\conssym)$ be an implementable solution with dispatch times $t\colon\conssym\to\NN$. Then, for $c=(a,\kappa)\in \conssym$ and $k\in\kappa$, $t(c)\in I_k(a)$.
\end{observation}
\begin{proof}
	Let $k\in\calK$ with $\spathsym_k = a_1\ldots a_n$ and $c_i=c(a_i,k)$.
	Then $t(c_1) \geq r_k = \edt{a_1}$ and, inductively, $t(c_{i+1}) \geq t(c_i) + \tau(a_i) \geq \edt{a_i} + \tau(a_i) \geq \edt{a_{i+1}}$ for $i\in (n)$.
	
	Similarly, $t(c_n) \leq \ell_k - \tau(a_n) = \ldt{a_n}$ and $t(c_i) \leq t(c_{i+1}) - \tau(a_i) \leq \ldt{a_{i+1}} - \tau(a_i) \leq \ldt{a_i}$ for $i\in(n)$ and the claim follows since for each $c=(a,\kappa)\in\conssym$ and $k\in\kappa$, $a\in\spathsym_k$.
\end{proof}

Returning to our example, we see that $\edt[k_2]{v_1} = 0$, $\edt[k_2]{v_2} = 2$, and $\edt[k_2]{v_3} = 4$.
Note that, if the network had shorter paths between these nodes, then these values would be lower, despite the solution using the longer paths.
In future illustrations, we only specify the bounds of interest and assume any omitted values are sufficiently loose, as may result from parts of the network that are not depicted.

%% file: sections/ddd-algo.tex
\section{A Dynamic Discretisation Discovery Algorithm}
\snd can be modelled as a mixed-integer program on a time-expanded network \cite{boland_continuous-time_2017}, though this quickly becomes intractable for instances with finer time resolutions \cite{boland_price_2019}.
Two alternative formulations that do not rely on a time expansion have been proposed \cite{hewitt_new_2023,shu_robust_2023}.
Both of them perform well on instances where commodities have only a small set of paths in the network to choose from.
But in general, the state-of-the-art approach for solving \snd is dynamic discretisation discovery with various later enhancements \parencite[e.g.,][]{boland_continuous-time_2017,helber_arc-based_2026,hewitt_enhanced_2019,marshall_interval-based_2021,shu_new_2026}.
In this section, we develop the algorithm outlined in~\cref{fig:ddd-flowchart} for \snd, which will form the basis for the rest of this work.

\subsection{The Relaxed Problem}
The centrepiece of a DDD algorithm is a relaxation of the problem, which should be significantly easier to solve.
We utilise a simplified version of the arc-based relaxation proposed by~\textcite{helber_arc-based_2026}, which is based on a discretisation of time for each arc in the network.

\begin{definition}
	A \emph{time discretisation} $\calT$ consists of a set $\Set{0,H}\subseteq \calT_a\subseteq[H]$ of time points for each arc $a\in \A{D}$.
	If $\calT_a = \Set{t_0,\ldots,t_n}$ with $0=t_0< t_1<\ldots<t_n = H$, then we define $\calT_a^I = \Set{[t_0,t_1),[t_1,t_2),\ldots,[t_{n-1},t_n)}$ as the intervals corresponding to the discretisation.
\end{definition}

The relaxed problem \rsnd{} is identical to \snd, with the exception that it takes a time discretisation \calT{} as an additional input and the notion of feasibility changes.
Specifically, instead of specifying the dispatch time for each consolidation, we only want to determine the time interval in which it is dispatched.
Formally, a solution $\solsym=(\pathsym,\conssym)$ is \emph{feasible} for R-SND or \emph{$\calT$-representable} if dispatch intervals $h\colon \conssym\to2^{\NN}$ exist with $h(c)\in\calT_{a(c)}^I$ such that, for all $k\in \calK$ with $\spathsym_k = a_1\ldots a_n$ and $c_i = c(a_i,k)$ for $i\in(n]$,
\begin{enumprop}
	\item $h(c_i)\cap I_k(a_i) \neq \emptyset$ for all $i\in(n]$, and \label{def:representable:nonemptyintervals}
	\item $\min h(c_i) + \tau(a_i) \leq \max h(c_{i+1})$ for all $i\in (n)$. \label{def:representable:linking}
\end{enumprop}
The first condition ensures that a commodity can only be part of a consolidation $c$ dispatched on arc $a$ in an interval $h(c)$ that at least overlaps with the interval $I_k(a)$ in which the commodity itself can traverse $a$.
The second condition ensures that consecutive consolidations of a commodity select consistent intervals.

If we recall our example in \cref{fig:snd-illustration}, then we could use the time discretisation $\calT_{v_1v_2} = \Set{0,1,H}$ and $\calT_{v_2v_3}=\Set{0,H}$.
With this discretisation, the solution is representable since we can assign the intervals $[1,H)$ and $[0,H)$ to the two consolidations, which satisfies both conditions.
This is an example of a discretisation that is too coarse to detect that the solution is non-implementable.
If we insert the time point $3$ into $\calT_{v_2v_3}$, then we must now select the interval $[0,3)$ for the second consolidation in order to satisfy the first condition since $I_{k_3}(v_2v_3)=[1,1]$.
The first consolidation must still use $[1,H)$ since $I_{k_1}(v_1v_2)=[1,1]$, so the second condition would require $1+2\leq 2$, which is false. Thus, adding this time point is enough to ensure that the solution also does not appear when solving the relaxation.

For the following discussion, it does not matter how exactly \rsnd is solved. 
But for completeness we give an integer program in \cref{appendix:IP}.

We do want to briefly note that the problem described here is, indeed, a relaxation.
\begin{lemma}
	\rsnd{} is a relaxation of \snd.
	For a full discretisation, it is tight.
\end{lemma}
\begin{proof}
	Let $\solsym=(\pathsym,\conssym)$ be a solution and $\calT$ be a time discretisation.
	If $\solsym$ is implementable, there exist dispatch times $t\colon\conssym\to\NN$.
	For $c\in\conssym$, we let $h(c)$ be the interval in $\calT^I_{a(c)}$ containing $t(c)$.
	
	Let $k\in\calK$ with $\spathsym_k = a_1\ldots a_n$ and $c_i = c(a_i,k)$ for $i\in(n]$.
	The first property holds since $t(c_i) \in h(c_i)\cap I_k(a_i)$ by \cref{intervals-implementable-dt}.
	Finally, since $t(c_{i+1})\geq t(c_i) + \tau(a_i)$, we get $\min h(c_i) + \tau(a_i) \leq \max h(c_{i+1})$ as well.
	Thus, $\solsym$ is $\calT$-representable.
	
	If $\calT$ is the full discretisation and $\solsym$ is $\calT$-representable, then we get single element dispatch intervals $h(c)$ for $c=(a,\kappa)\in\conssym$ and we set $t(c)$ to the single element of $h(c)$.
	Since $h(c,k) \coloneqq h(c) \cap I_k(a) \neq \emptyset$ for all $k\in\kappa$, $t(c)$ is actually equal to the element in $h(c,k)$ as well. 
	
	Again, let $k\in\calK$ with $\spathsym_k = a_1\ldots a_n$ and $c_i = c(a_i,k)$ for $i\in(n]$.
	Then $t(c_1) = \min h(c_1,k) \geq \edt{a_1} = \edt{o_k} = r_k$, $t(c_{i+1}) = \max h(c_{i+1}) \geq \min h(c_{i}) + \tau(a_i) = t(c_i) + \tau(a_i)$, and $t(c_n) + \tau(a_n) = \max h(c_n,k) + \tau(a_n) \leq \ldt[k]{a_n} + \tau(a_n) = \ldt{d_k} = \ell_k$.
\end{proof}

\subsection{Implementable and Representable Solutions}
As we saw in the example, a solution $\solsym$ to \rsnd for some given discretisation $\calT$ need not be implementable. 
Recall that in this case we want to identify \emph{conflicts}, which are some structures of the solution that stand in the way of implementability, and \emph{fixes}, which are changes to the relaxation that prevent these conflicts from recurring.
The conflicts we study are the so-called \emph{too-long paths} in the \emph{dispatch-node graph} introduced by \textcite[][Definition 1]{shu_new_2026}. 
This graph contains nodes $\vdn{v}{k}$ for every node $v$ in each commodity's path $\spathsym_k$, and arcs between two nodes $\vdn{v}{k}$ and $\vdn{v'}{k'}$ exactly when commodities $k$ and $k'$ are consolidated on arc $vv'$.
\begin{definition}[Dispatch-node graph]
	Given a solution $\solsym=(\pathsym,\conssym)$, its \emph{dispatch-node graph} $D_\solsym$ is given by
	\begin{align*}
		\V{D_\solsym} &= \Set{\vdn{v}{k} : k\in\calK,\, v\in \spathsym_k}, \\ 
		\A{D_\solsym} &= \Set{\vdn{v}{k}\vdn{v'}{k'} : (vv',\kappa)\in \conssym,\, k,\,k'\in \kappa}.
	\end{align*}
	For $b=\vdn{v}{k}\vdn{v'}{k'}\in\A{D_\solsym}$, we write $a(b)=vv'$ for its underlying arc in $D$, and set $\tau(b) = \tau(a(b))$ and $\calT_b = \calT_{a(b)}$.
	For $u=\vdn{v}{k}\in\V{D_\solsym}$, we set $\edt[]{u} = \edt{v}$ and $\ldt[]{u}=\ldt{v}$.
\end{definition}
The dispatch-node graph for our running example from \cref{fig:snd-illustration} is shown in \cref{fig:dispatch-node-graph-illustration}.
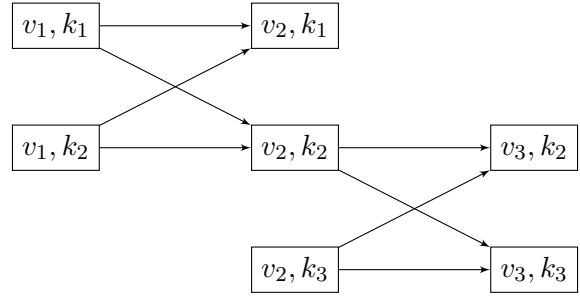
\begin{figure}
	\centering
	\begin{tikzpicture}
		\node[draw] (v1k1) {$v_1,k_1$};
		\node[draw,below=1 of v1k1] (v1k2) {$v_1,k_2$};
		\node[draw,right=2 of v1k1] (v2k1) {$v_2,k_1$};
		\node[draw,right=2 of v1k2] (v2k2) {$v_2,k_2$};
		\node[draw,below=1 of v2k2] (v2k3) {$v_2,k_3$};
		\node[draw,right=2 of v2k2] (v3k2) {$v_3,k_2$};
		\node[draw,right=2 of v2k3] (v3k3) {$v_3,k_3$};
		
		\draw[->] (v1k1) to (v2k1);
		\draw[->] (v1k1) to (v2k2);
		\draw[->] (v1k2) to (v2k1);
		\draw[->] (v1k2) to (v2k2);
		
		\draw[->] (v2k2) to (v3k2);
		\draw[->] (v2k2) to (v3k3);
		\draw[->] (v2k3) to (v3k2);
		\draw[->] (v2k3) to (v3k3);
	\end{tikzpicture}
	\caption{The dispatch-node graph for the example from \cref{fig:snd-illustration}.
	\label{fig:dispatch-node-graph-illustration}}
\end{figure}

Looking at the example, we observe the following simple but useful property of the neighbourhood of a node in~$D_\solsym$.
\begin{observation}
	\label{ndg-in-out}
	Let $u=\vdn{v}{k}\in\V{D_\solsym}$. Then $v\in\spathsym_k$.
	Note that any arc $b=\vdn{v'}{k'}\vdn{v}{k}\in\Ain{u}$ originates from a consolidation $c=(v'v,\kappa)$ with $k\in\kappa$.
	By the definition of a solution, $v'v\in\spathsym_k$. In particular, if $v=o_k$, then $\Nin{u} = \emptyset$.
	Otherwise, let $v'$ be the predecessor of $v$ in $\spathsym_k$.
	Then $c=c(v'v,k)$ and $\Nin{u} = \Set{\vdn{v'}{k'} : k'\in\kappa}$.
	
	Similarly, $\Nout{u} = \emptyset$ if $v=d_k$ and otherwise $\Nout{u} = \Set{\vdn{v'}{k'} : k'\in\kappa}$, where $v'$ is the successor of $v$ in $\spathsym_k$ and $c(vv',k) = (vv',\kappa)$.
\end{observation}

As described by~\textcite[][Section 7.2]{shu_new_2026}, a walk $\walksym$ from $u=\vdn{v}{k}$ to $u'=\vdn{v'}{k'}$ in $D_\solsym$ gives a lower bound $\edt[]{u} + \tau(\walksym)$ on the arrival time of $k'$ at $v'$ under the given sequence of consolidations.
If this is later than the latest time $\ldt[]{u'}$ at which $k'$ can be dispatched from $v'$, this evidently contradicts implementability of $\solsym$.
In \cref{fig:dispatch-node-graph-illustration}, this happens for the path $\vdn{v_2}{k_2}\vdn{v_3}{k_3}=u_1u_2$, which has $\edt[]{u_1}=2$ and $\ldt[]{u_2} = 3$, so its real arrival time is $2+2=4>3$.
This motivates the following definition:
\begin{definition}[Real arrival time and too-long paths]
	For a walk $\walksym=u_0\ldots u_n$ in $D_\solsym$ and $i\in [n]$, we define the real arrival time $\rdt{i} = \edt[]{u_0} + \tau(\walksym u_i)$ and call $\walksym$ a too-long path if $\rdt{n} > \ldt[]{u_n}$.
\end{definition}

\begin{remark}
    This definition differs slightly from that of~\textcite[][Definition 2]{shu_new_2026}; they require that a too-long path starts at a commodity's origin, i.e., $u = (o_k, k)$ for some $k \in \calK$. 
    In most respects, both definitions yield essentially the same results, but the subsequent discussion is easier to follow with our definition, since it allows us to disregard various edge cases.
	We note that to stay consistent with their terminology, we say too-long path even though walks are also permitted.
\end{remark}

Despite our slightly different definition of a too-long path, the following result transfers directly:
\begin{theorem}[{{\cite[][Theorem 2]{shu_new_2026}}}]
	\label{characterisation:implementable}
	A solution $\solsym$ is implementable if and only if the corresponding dispatch-node graph $D_\solsym$ contains no too-long path.
\end{theorem}
\begin{proof}
	Let $\solsym=(\pathsym,\conssym)$ be implementable with dispatch times $t\colon\conssym\to\NN$.
	Moreover, let $\walksym=u_0\ldots u_n$ be a walk in $D_\solsym$ with $u_i=\vdn{v_i}{k_i}$, for $i\in[n]$ and $a_i=v_{i-1}v_i$, and $c_i = c(a_i,k_i)$ for $i\in(n]$.
	
	Since, for $i\in(n)$, $c_{i}$ and $c_{i+1}$ are consolidations on consecutive arcs of $\spathsym_{k_{i}}$, we get $t(c_{i+1}) \geq t(c_{i}) + \tau(a_i)$ and, thus, $t(c_n) + \tau(a_n) \geq t(c_1) + \tau(\walksym)$.
	By \cref{intervals-implementable-dt}, $t({c_1}) \geq \edt[k_0]{a_1} =  \edt[k_0]{v_0} = \edt[]{u_0}$ and $t(c_n)\leq \ldt[k_{n}]{a_n} = \ldt[k_n]{v_n} - \tau(a_n) = \ldt[]{u_n} - \tau(a_n)$, so $\walksym$ is not too-long.
	
	Conversely, assume that $D_\solsym$ has no too-long path.
	Then, in particular, $D_\solsym$ is acyclic.
	By \cref{ndg-in-out}, the nodes without incoming arcs are exactly the nodes $\vdn{o_k}{k}$ for $k\in\calK$, for which we define $t(\vdn{o_k}{k}) = r_k$.
	For the remaining nodes $u$, we set $t(u) = \max\Set{t(u') + \tau(u'u) : u'\in\Nin{u}}$, which is well-defined as $D_\solsym$ is acyclic.
	Finally, for $c=(vv',\kappa)\in\conssym$, we define $t(c) = \max\Set{t(\vdn{v}{k}) : k\in\kappa}$.
	
	We first note that, by \cref{ndg-in-out}, $\Nin{\vdn{v}{k}} = \Set{\vdn{v'}{k'} : k'\in \kappa}$, where $v'$ is the predecessor of $v$ in $\spathsym_k$ and $c=c(v'v,k) = (v'v,\kappa)$.
	In particular, $t(u) = t(c) + \tau(v'v)$.
	
	We now simply verify that $t$ satisfies the desired properties.
	Let $k\in\calK$ with $\spathsym_k = a_1\ldots a_n$, $c_i = c(a_i,k) = (a_i,\kappa_i)$, and $a_i=v_{i-1}v_i$ for $i\in(n]$.
	Then $t(c_1) \geq t(\vdn{o_k}{k}) = r_k$ and, for $i\in(n)$, $t(c_{i+1}) \geq t(\vdn{v_i}{k}) = t(c_i) + \tau(a_i)$.
	
	Finally, we show that for each $u\in D_\solsym$ there exists a $u'u$-path $\walksym$ such that $t(u)\leq \edt[]{u'} + \tau(\walksym)$.
	This implies the final property, since applied to $u=\vdn{v_n}{k}=\vdn{d_k}{k}$ we obtain $t(c_n) + \tau(a_n) = t(u) \leq \edt[]{u'} + \tau(\walksym) \leq \ldt[]{u} = \ell_k$.
	
	To see that such paths exist, we use the topological ordering we get from acyclicity.
	For the induction start, we know that $u=\vdn{o_k}{k}$ for some $k\in\calK$, for which $t(u) = r_k = \edt[]{u}$, so we can use the trivial path $\walksym=u$.
	In the induction step, we get that $t(u) = t(u'')+ \tau(u''u)$ for some $u''\in\Nin{u}$.
	By the induction hypothesis, we get a $u'u''$-path $\walksym'$ satisfying $t(u'') \leq \edt[]{u'} + \tau(\walksym')$.
	We extend this path by the arc $u''u$ to get a $u'u$-path~$\walksym$ satisfying $t(u) = t(u'')+\tau(u''u) \leq \edt[]{u'} + \tau(\walksym)$ as required.
\end{proof}

Note that the above construction directly gives a feasible timing if $\solsym$ is implementable.
Otherwise, we have identified a too-long path in the dispatch-node graph, and now need to find a fix for it, that is, an adjustment of the discretisation that gets rid of it.
More formally:
\begin{definition}
	Let $\solsym$ be a solution, $\walksym$ be a walk in $D_\solsym$, and $\calT$ be a time discretisation.
	We say $\walksym$ is \emph{eliminated} by $\calT$ if there exists no $\calT$-representable solution $\solsym'$ whose dispatch-node graph $D_{\solsym'}$ contains $\walksym$.
\end{definition}

In fact, since any too-long path makes a solution non-implementable, we want to eliminate \emph{all} too-long paths in $D_\solsym$.
Luckily, we get many of them for free, thanks to the following observation \cite[][Section 7.2]{shu_new_2026}.
\begin{observation}\label{obs:elimination_of_extensions}
	If a walk $\walksym$ is eliminated by $\calT$, then all extensions $\walksym'$ of $\walksym$ are also eliminated by $\calT$.
\end{observation}
\begin{proof}
	If $\walksym'$ is not eliminated, there exists a solution $\solsym'$ that is $\calT$-representable such that $\walksym'$ appears in $D_{\solsym'}$.
	But then $\walksym$ also appears in $D_{\solsym'}$.
\end{proof}

This motivates the following definition, similar to Definition 3 of \cite{shu_new_2026}, but again not requiring too-long paths to start at a commodity's origin:
\begin{definition}[Minimal too-long path]
	A too-long path is \emph{minimal} if none of its proper sub-walks are too-long.
\end{definition}
\begin{corollary}
	A solution $\solsym$ is implementable if and only if the corresponding dispatch-node graph $D_\solsym$ contains no minimal too-long path.
\end{corollary}
Hence, it is, in fact, sufficient to eliminate all minimal too-long paths.
In our example from \cref{fig:snd-illustration,fig:dispatch-node-graph-illustration}, we already saw that $\walksym=\vdn{v_2}{k_2}\vdn{v_3}{k_3}$ is too-long; it is the unique minimal too-long path.
The walk $\walksym'=\vdn{v_1}{k_1}\vdn{v_2}{k_2}\vdn{v_3}{k_3}$ is also too-long and its real arrival time is $1+2+2=5>3$, but since it contains $\walksym$ as a subpath, it suffices to eliminate $\walksym$.

To actually eliminate the minimal too-long paths, we characterise representability in terms of too-long paths as well and use this to determine a sufficient criterion for a too-long path $\walksym$ to be eliminated.
Essentially, the idea is to check if the consolidations implied by $\walksym$ cause the last commodity to arrive too late even in the relaxation.
Since the relaxation specifies intervals instead of time points and the consistency requirement just asks consecutive intervals to contain time points that would allow the transition, we always select the earliest possible interval.

Returning to our example in \cref{fig:snd-illustration,fig:dispatch-node-graph-illustration} and using the same time discretisation $\calT_{v_1v_2} = \Set{0,1,H}$ and $\calT_{v_2v_3}=\Set{0,3,H}$ as before, the walk $\walksym=\vdn{v_2}{k_2}\vdn{v_3}{k_3}=u_1u_2$ is unproblematic for the relaxation.
We are available at time $\edt[]{u_1}=2$ at $u_1$, so the interval $[0,3)$ would be an option.
Since the relaxation does not distinguish between the points in this interval, we can take the earliest point in the interval, namely its lower bound, letting us (fictitiously) arrive at $u_2$ at time $2$, which is early enough to not conflict with $\ldt[]{u_2} = 3$.
Indeed, $\walksym$ is not eliminated: splitting $c_1$ into singleton consolidations allows us to assign $[1,H)$ to $k_1$ and $[0,1)$ to $k_2$ on the first arc, and $[0,3)$ to $c_2$, giving a representable solution that still contains $\walksym$.

However, if we extend this path backwards by $u_0 = \vdn{v_1}{k_1}$, then starting at $\edt[]{u_0}=1$ and using $[1,H)$ on $v_1v_2$ gives a relaxed arrival time of $3$ at $u_1$.
This requires us to use the interval $[3,H)$ on $v_2v_3$, so, even in the relaxed problem, we cannot arrive before time $5$ at $u_2$, which is too late since $\ldt[]{u_2} = 3$.
The following definition formalises this computation of the relaxed earliest arrival time.
\begin{definition}[Relaxed arrival times and relaxed too-long paths]
	\label{def:relaxed_arrival_times}
	Let $\walksym=u_0b_1u_1\ldots b_nu_n$ be a walk in $D_\solsym$ and $\calT$ be a time discretisation.
	We define the \emph{relaxed arrival time $\rat{i}$} at the $i$th node of $\walksym$ by
	\begin{displaymath}
		\rat{i} = 
		\begin{cases}
			\edt[]{u_0} & \text{if $i=0$,}\\
			\rounddown{\rat{i-1}}{b_{i}} +\tau(b_{i}) & \text{otherwise}
		\end{cases}
	\end{displaymath}
	where $\rounddown{t}{a} = \max\Set{t'\in\calT_a : t'\leq t}$.
	If $\rat{n} > \ldt[]{u_n}$, we say \(\walksym\) is $\calT$-\emph{relaxed too-long}.
\end{definition}
\begin{remark}\label{remark:relaxed_arrival_time_definition}
	In the short paper~\cite{Wullner2026-xl} we say a too-long path is eliminated if and only if it fulfils this criterion, which is not consistent with the terminology in this paper.
	The idea behind this criterion was already used by~\textcite[][Theorem 3]{helber_arc-based_2026} to check if too-long paths are already eliminated.
	A stronger variant of the criterion is also introduced in the newer version of the preprint~\textcite[][Lemma~2]{shu_new_2026}.
	We discuss how our contributions can be adapted to their criterion in~\cref{sec:variants}, but do not use it for our exposition since it increases notational burden and edge cases without yielding more interesting results.
\end{remark}

Conveniently, relaxed too-long paths are obstacles to representability, just as too-long paths are to implementability.
\begin{theorem}
	\label{characterisation:representable}
	Let $\calT$ be a time discretisation.
	A solution $\solsym$ is $\calT$-representable if and only if $D_\solsym$ contains no $\calT$-relaxed too-long path.
\end{theorem}
\begin{proof}
	Let $\solsym=(\pathsym,\conssym)$ be $\calT$-representable with dispatch intervals $h\colon \conssym\to2^\NN$.
	Moreover, let $\walksym=u_0\ldots u_n$ be a walk in $D_\solsym$ with $u_i = \vdn{v_i}{k_i}$ for $i\in[n]$ and $a_i=v_{i-1}v_i$, $c_i=c(a_i,k_i)=c(a_i,k_{i-1})$ for $i\in(n]$.
	
	We show that $\rat{i-1}\leq \max h(c_{i})$, for all $i\in(n]$.
	For $i=1$, since $h(c_1)\cap I_{k_0}(a_1)\neq\emptyset$, $\max h(c_1) \geq \edt[k_0]{a_1} = \edt[]{u_0} = \rat{0}$.
	For $i\in(n)$, using the induction hypothesis and the fact that $\solsym$ is $\calT$-representable, $\rat{i} = \rounddown{\rat{i-1}}{a_i}  + \tau(a_{i})\leq \rounddown{\max h(c_{i})}{a_i} + \tau(a_{i}) = \min h(c_{i}) + \tau(a_i) \leq \max h(c_{i+1})$.
	
	Since $h(c_n)\cap I_{k_n}(a_n)\neq\emptyset$, $\min h(c_n)\leq \ldt[k_n]{a_n} = \ldt[]{u_n}-\tau(a_n)$ and, consequently, $\rat{n} = \rounddown{\rat{n-1}}{a_n} + \tau(a_n) \leq \rounddown{\max h(c_{n})}{a_n} + \tau(a_n) = \min h(c_n) + \tau(a_n) \leq \ldt[]{u_n}$.
	Thus, $\walksym$ is not relaxed too-long.
	
	For the converse, we assume that $D_\solsym$ contains no relaxed too-long path.
	This lets us define a dispatch time at every node $u$ in $D_\solsym$ by setting
	\begin{displaymath}
		t(u)=\max\Set{\rat{n} : \walksym \text{ is a walk in $D_\solsym$ ending at $u$}}.
	\end{displaymath}
	Note that we explicitly allow walks of length 0 here, meaning $t(u) \geq \edt[]{u}$.
	Also, since no relaxed too-long paths exist, this value is finite and satisfies $t(u) \leq \ldt[]{u}$.
	For $c=(vv',\kappa)$, we let $t(c)=\max\Set{t(\vdn{v}{k}) : k\in\kappa}$ and $h(c)$ be the interval in $\calT^I_{vv'}$ containing $t(c)$.
	We use this to show that $\solsym$ is $\calT$-representable.
	
	To this end, let $k\in \calK$ with $\spathsym_k = a_1\ldots a_n$, $c_i = c(a_i,k)=(a_i,\kappa_i)$, and $a_i=v_{i-1}v_i$ for $i\in(n]$.
	Furthermore, let $u_i=\vdn{v_i}{k}$ for $i\in[n]$.
	By definition, for $i\in(n]$, $\max h(c_i) \geq t(c_i) \geq t(u_{i-1}) \geq \edt[]{u_{i-1}} = \edt{a_i}$.
	
	Again, let $i\in (n]$, and let $u'=(v_{i-1},k')$ with $k'\in\kappa_i$ such that $t(c_i) = t(u')$.
	Moreover, let $\walksym$ be a walk of length $m\geq 0$ in $D_\solsym$ ending at $u'$ with $\rat{m} = t(u') = t(c_i)$.
	By extending $\walksym$ to $u_i$, we get a walk $\walksym'$ ending at $u_i$ and, thus, $t(u_i) \geq \rat[\walksym']{m+1} = \rounddown{t(c_i)}{a_i} + \tau(a_i)$.
	Since $\walksym'$ is not relaxed too-long, $\rat[\walksym']{m+1} \leq \ldt[]{u_i} = \ldt{a_i}+\tau(a_i)$.
	Hence, $\min h(c_i) = \rounddown{t(c_i)}{a_i} \leq \ldt{a_i}$ and $h(c_i)\cap I_k(a_i)\neq \emptyset$.

	For the final property, let $i\in(n)$.
	Using the above computations, we see that $\min h(c_i) + \tau(a_i) = \rounddown{t(c_i)}{a_i} + \tau(a_i) \leq t(u_i) \leq t(c_{i+1}) \leq \max h(c_{i+1})$ and the proof is complete.
\end{proof}
We point out that we cannot restrict ourselves to minimal too-long paths in~\cref{characterisation:representable}, since our example showed that a solution can be non-representable because a non-minimal too-long path is relaxed too-long, but none of the minimal ones are.

Evidently, if we want to make a non-implementable solution non-representable, we need to select a discretisation that makes at least one walk $\walksym$ in $D_\solsym$ relaxed too-long.
Since, by definition, $\rat{i}\leq \rdt{i}$ for all $i$, such a walk must be too-long.
Furthermore, relaxed too-long paths are guaranteed not to reappear in representable solutions, leading to the following natural criterion for elimination.
\begin{observation}\label{obs:eliminated_if_rtl}
	A too-long path $\walksym$ is eliminated by a discretisation $\calT$ if it is $\calT$-relaxed too-long.
\end{observation}

\begin{example}
We demonstrate that the converse of~\cref{obs:eliminated_if_rtl} is not true; a too-long path may be eliminated by \(\calT\) even if it is not \(\calT\)-relaxed too-long.
	For non-minimal paths this follows directly from~\cref{obs:elimination_of_extensions}, but it is not at all obvious that the same is true for minimal paths.
	Consider the example in \cref{fig:eliminated_but_not_rtl}, especially the minimal too-long path $\walksym$.
	We select a discretisation $\calT$ that is full (i.e., $\calT_a = [H]$) except that $5 \notin \calT_{v_2v_3}$ and $6 \notin \calT_{v_3v_4}$.
	Then $\walksym$ is not relaxed too-long, since $\rat{0} = 2, \rat{1} = 3, \rat{2} = 5, \rat{3} = 6, \rat{4} = 7$, and $\rat{5} = 9 \not> 9 = \ell_{k_1}$.
	For this path to exist in a solution's dispatch-node graph, commodity $k_1$ must traverse $v_0v_1$ together with $k_3$, as well as $v_1v_2$ and $v_4v_5$.
	Taking $v_2v_5$ would prevent it from traversing the required arc $v_4v_5$, so it must take $v_2v_3$.
	This implies that $\walksym'$ also exists in the dispatch-node graph, and this path is relaxed too-long, since $\rat[\walksym']{0} = 2, \rat[\walksym']{1} = 3, \rat[\walksym']{2} = 5, \rat[\walksym']{3} = 6 > 5 = \ldt[k_1]{v_3}$.
	Since existence of $\walksym$ implies $\walksym'$ and this path is relaxed too-long, both of them are eliminated.
	
	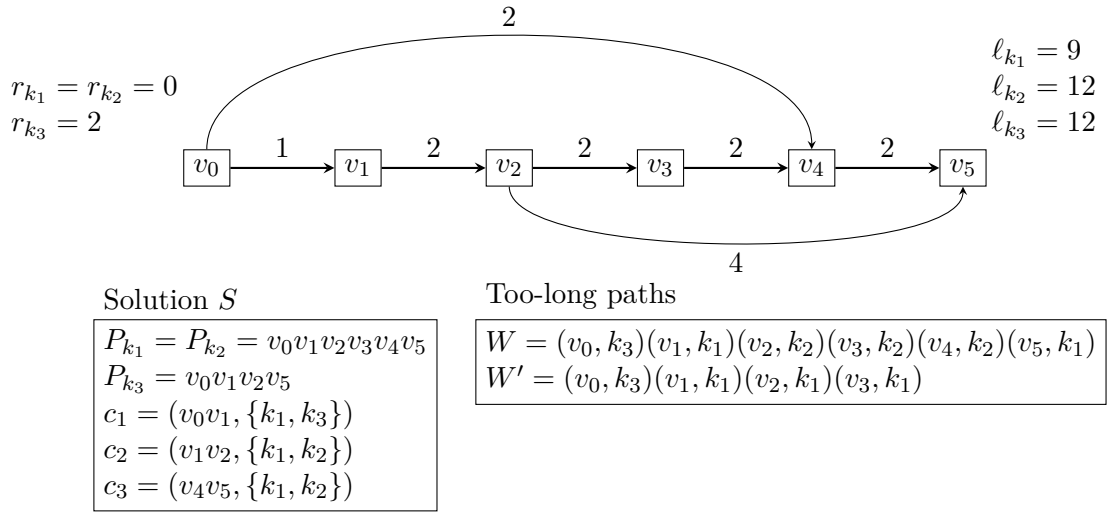
\begin{figure}[htbp]
		\centering
		\begin{tikzpicture}[>=stealth]
			\node[draw, label={[align=left]above left:{$r_{k_1}=r_{k_2}=0$\\$r_{k_3}=2$}}] (v0) at (0,0) {$v_0$};
			\node[draw] (v1) at (2,0) {$v_1$};
			\node[draw] (v2) at (4,0) {$v_2$};
			\node[draw] (v3) at (6,0) {$v_3$};
			\node[draw] (v4) at (8,0) {$v_4$};
			\node[draw, label={[align=left]above right:{$\ell_{k_1}=9$\\$\ell_{k_2}=12$\\$\ell_{k_3}=12$}}] (v5) at (10,0) {$v_5$};
			
\draw[->, thick] (v0) to node[above]{1} (v1);
			\draw[->, thick] (v1) to node[above]{2} (v2);
			\draw[->, thick] (v2) to node[above]{2} (v3);
			\draw[->, thick] (v3) to node[above]{2} (v4);
			\draw[->, thick] (v4) to node[above]{2} (v5);
			\draw[->] (v0.north) .. controls +(0,2) and +(0,2) ..
			node[above]{2} (v4.north);
			\draw[->] (v2.south) .. controls +(0,-1) and +(0,-1) ..
			node[below]{4} (v5.south);
			
			\node[align=left, anchor=north west, draw, label={[shift=(solbox.north west)]above right:Solution $\solsym$}] (solbox) at (-1.5,-2) {
				$\spathsym_{k_1}=\spathsym_{k_2}=v_0v_1v_2v_3v_4v_5$\\
				$\spathsym_{k_3}=v_0v_1v_2v_5$\\
				$c_1 = (v_0v_1, \Set{k_1,k_3})$\\
				$c_2 = (v_1v_2, \Set{k_1,k_2})$\\
				$c_3 = (v_4v_5, \Set{k_1,k_2})$
			};
			\node[align=left, label={[shift=(tlps.north west)]above right:Too-long paths}, draw, anchor=north west, xshift=5mm] (tlps) at (solbox.north east) {
				$\walksym=(v_0,k_3)(v_1,k_1)(v_2,k_2)(v_3,k_2)(v_4,k_2)(v_5,k_1)$\\
				$\walksym' = (v_0,k_3)(v_1,k_1)(v_2,k_1)(v_3,k_1)$
			};
		\end{tikzpicture}
		\caption{Example of a minimal too-long path that is eliminated by $\calT$ but not $\calT$-relaxed too-long. Arcs are labelled with travel time. Commodities' origins and destinations are labelled with their release time and deadline.
		All unlisted consolidations are singletons.
		}
		\label{fig:eliminated_but_not_rtl}
	\end{figure}
\end{example}
So we see that selecting $\calT$ so that all minimal too-long paths are $\calT$-relaxed too-long may be too conservative; a smaller discretisation may already eliminate them due to additional routing considerations.
Finding a practical characterisation of an eliminated too-long path is an interesting open question.
But we suspect that in practice, the increase in discretisation size from using our sufficient characterisation will usually be negligible.

%% file: sections/refinement.tex
\section{The Optimal Refinement Problem}

Now that we have identified conflicts (minimal too-long paths) and potential fixes (making them relaxed too-long), we can now define the problem studied in the rest of this work:
Given a solution $\solsym$ to \snd, the \emph{optimal refinement problem} (ORP) is to find a discretisation \(\calT\) of minimum size $\textstyle\sum_{a \in \A{D}} \abs{\calT_a}$ such that all minimal too-long paths in $D_\solsym$ are $\calT$-relaxed too-long.

The objective of minimising the number of time points is motivated by the wish to reduce the size of the integer programming models used to solve \rsnd.
\begin{remark}
	In the short paper \cite{Wullner2026-xl}, we show that under this objective function, the problem of making an arbitrary set of too-long paths also relaxed too-long is NP-hard. 
	We note that it is possible to construct instances of \snd such that the hardness result transfers to ORP as stated here, but omit the construction, which provides little additional insight.
\end{remark}

This section is devoted to solving the ORP.
First, we look at properties of minimal too-long paths in \cref{subsec:min-tlp} before showing how the ORP can be solved if one has explicitly enumerated all minimal too-long paths in \cref{subsec:pathbasedapproaches}.
An algorithm that enumerates all these paths can be found in Section 7.4 of~\cite{shu_new_2026} (and can be easily adapted to our definition of a too-long path).
We then demonstrate that this approach can be unattractive since the number of such paths may be exponential in instance parameters.
This motivates us to develop a more compact representation of all minimal too-long paths in \cref{subsec:graph-criterion} and corresponding solution approaches in \cref{subsec:constructing-tetlpg,subsec:solving-ORP-graph,subsec:solving-ORP-DDD}.

For the remainder of the section, let $\calT$ denote a time discretisation and $\tlpsymb$ the set of all minimal too-long paths in $D_\solsym$.

\subsection{Properties of Minimal Too-Long Paths}\label{subsec:min-tlp}
For everything that happens later in this section, from formulating an integer program for the ORP to developing a sufficient criterion for eliminating all minimal too-long paths, it is helpful to know certain properties that these minimal too-long paths have.
We use this subsection to collect several such properties.

First, we note that if node $v$ is visited by commodity $k$, then this commodity needs to arrive at $v$ at some point in the interval $[\edt{v},\ldt{v}]$.
Thus, this interval describes a certain amount of slack that is available at the node.
If we deviate from the lower bound $\edt{v}$ we have used up some of the available slack, and if we exceed the upper bound $\ldt{v}$, then we are late, which happens at the end of too-long paths.
Let us formally define these values.
\begin{definition}[Slack consumption and lateness]
	Let $u\in\V{D_\solsym}$, $t\in \NN$, and $w=(u,t)$.
	We define the \emph{slack consumption} and \emph{lateness} at $w$ by 
	\begin{equation*}
		\delay(w) = t - \edt[]{u}, \qquad \exc(w) = t - \ldt[]{u}.
	\end{equation*}
\end{definition}

Minimal too-long paths are only late at their last node.
Since starting such a path at a later node lets it finish earlier by exactly its slack consumption, this value must be at least equal to the lateness at all but the first node (where it is zero), otherwise the path would not be minimal.
Formally, we obtain the following observation.
\begin{observation}
	\label{used-buffer-lateness-minimal-tlp}
	If $\walksym=u_0\ldots u_n$ is a minimal too-long path and $w_i=(u_i,\rdt{i})$, then the following properties hold:
	\begin{itemize}
		\item $\exc(w_n) > 0$ and $\exc(w_i)\leq 0$ for all $i\in[n)$.
		\item $\delay(w_0) = 0$ and $\delay(w_i) \geq \exc(w_n) > 0$ for all $i\in(n]$.
	\end{itemize}
\end{observation}
\begin{proof}
	For the lateness, $\exc(w_n) > 0$ since $\walksym$ is too-long and, for $i\in[n)$, $\exc(w_i) \leq 0$ since $\walksym u_i$ is not too-long.
	
	For the slack consumption, $\delay(w_0) = 0$ since $\rdt{0} = \edt[]{u_0}$, and, for $i\in(n]$, $u_i\walksym$ is not too-long, so $\edt[]{u_i} + \tau(u_i\walksym) \leq \ldt[]{u_n} = \rdt{n} - \exc(w_n) = \edt[]{u_0} + \tau(\walksym) -\exc(w_n) = \rdt{i} + \tau(u_i\walksym) - \exc(w_n)$.
	Thus, $\rdt{i} \geq \edt[]{u_i} +\exc(w_n)$ and $\delay(w_i) \geq \exc(w_n)$.
\end{proof}

In fact, these properties are actually equivalent.
\begin{lemma}
	\label{used-buffer-lateness-minimal-tlp-equiv}
	If $\walksym=u_0\ldots u_n$ is a walk in $D_\solsym$ and $w_i=(u_i,\rdt{i})$, then $\walksym$ is minimal too-long if and only if
	\begin{itemize}
		\item $\exc(w_i)\leq 0$ for all $i\in[n)$ and
		\item $\delay(w_i) \geq \exc(w_n) > 0$ for all $i\in(n]$.
	\end{itemize}
\end{lemma}
\begin{proof}
	The forward implication follows from \cref{used-buffer-lateness-minimal-tlp}.
	For the converse, let $\walksym=u_0\ldots u_n$ be a walk satisfying the given properties and $w_i=(u_i,\rdt{i})$.
	Since $\exc(w_n)>0$, $\walksym$ is a too-long path.
	Let $u_i\walksym u_j$ be a sub-walk of $\walksym$ that is too-long.
	If $j<n$, then $\edt[]{u_i} + \tau(u_i\walksym u_j) \leq \edt[]{u_0} + \tau(u_0\walksym u_j) = \rdt{j} = \exc(w_j) + \ldt[]{u_j} \leq \ldt[]{u_j}$.
	Thus, $j=n$.
	
	If $i>0$, then $\edt[]{u_i} + \tau(u_i\walksym) = \rdt{i} - \delay(w_i) + \tau(u_i\walksym) = \edt[]{u_0} + \tau(\walksym) - \delay(w_i) \leq \rdt{n} -\exc(w_n) = \ldt[]{u_n}$.
	Hence, $i=0$ and $\walksym$ is minimal.
\end{proof}

Another useful observation is that any minimal too-long path that is also relaxed too-long must have relaxed arrival times that exceed the earliest arrival times at all but the first node.
\begin{observation}
	\label{minimal-tlp-and-rtl-properties}
	Let $\calT$ be a time discretisation and $\walksym=u_0\ldots u_n$ be a minimal too-long path that is $\calT$-relaxed too-long.
	Then $\rat{i} > \edt[]{u_i}$ for all $i>0$.
\end{observation}
\begin{proof}
	Let $\walksym=u_0\ldots u_n$ be a minimal too-long path that is $\calT$-relaxed too-long and $i>0$.
	Since $u_i\walksym$ is not too-long, $\edt[]{u_i} + \tau(u_i\walksym) \leq \ldt[]{u_n} < \rat{n} \leq \rat{i} + \tau(u_i\walksym)$, so $\edt[]{u_i} < \rat{i}$.
\end{proof}

\subsection{Path-Based Approaches}\label{subsec:pathbasedapproaches}

We want to formulate the ORP as a binary integer program, which can already be found in our short paper \cite[Section~4]{Wullner2026-xl}.
For the upcoming IP-formulation it is helpful to have an alternative characterisation for when a too-long path is actually relaxed too-long.
\begin{lemma}\label{lemma:relaxedtoolongtimepoints}
	A too-long path $\walksym = u_0b_1u_1\ldots b_nu_n$ is $\calT$-relaxed too-long if and only if for every $i \in (n]$ there is a time point $t_i \in \calT_{b_i}$ such that
	\begin{enumprop}
		\item $t_1 \leq \edt[]{u_0}$,\label{lemma:relaxedtoolongtimepoints:start}
		\item $t_i + \tau(b_i) \geq t_{i+1}$ for all $i \in (n)$, \label{lemma:relaxedtoolongtimepoints:linking}
		\item and $t_n + \tau(b_n) > \ldt[]{u_n}$.\label{lemma:relaxedtoolongtimepoints:end}
	\end{enumprop}
	Additionally, $t_i \in I_\walksym(i) = [\rdt{i-1} - \exc(u_n,\rdt{n}) + 1, \rdt{i-1}]$ for all $i \in (n]$.
\end{lemma}
\begin{proof}
	Evidently, $\rat{i} - \tau(b_i)\in \calT_{b_i}$ for $i > 0$.
	Assume $\walksym$ is relaxed too-long. Then we can set $t_i = \rat{i} - \tau(b_i)$.
	Then, by definition of $\rat{i}$, all three conditions hold.
	
	Conversely, assume that $t_i$ fulfilling the conditions exist. 
	Inductively, $\rat{i-1} \geq t_i$: the base case follows from~\cref{lemma:relaxedtoolongtimepoints:start}, and since $t_i\in\calT_{b_i}$, we obtain $\rat{i} = \rounddown{\rat{i-1}}{b_i} + \tau(b_i) \geq t_i + \tau(b_i)$.
	For $i<n$, this is at least $t_{i+1}$ by~\cref{lemma:relaxedtoolongtimepoints:linking}; for $i=n$, it is greater than $\ldt[]{u_n}$ by~\cref{lemma:relaxedtoolongtimepoints:end}.
	Thus, $\walksym$ is relaxed too-long.
	
	For the additional statement, observe that \eqref{lemma:relaxedtoolongtimepoints:start} and \eqref{lemma:relaxedtoolongtimepoints:linking} together give
	\begin{equation*}
		t_i \leq t_1 + \sum_{j = 1}^{i-1}\tau(b_j) \leq \edt[]{u_0} + \sum_{j =1}^{i-1}\tau(b_j) = \rdt{i-1}.
	\end{equation*}
	Similarly, \cref{lemma:relaxedtoolongtimepoints:linking,lemma:relaxedtoolongtimepoints:end} give
	\begin{equation*}
		t_i + \sum_{j=i}^n \tau(b_j) \geq t_n + \tau(b_n) > \ldt[]{u_n}
	\end{equation*}
	and therefore
	\begin{equation*}
		t_i > \ldt[]{u_n} - \sum_{j=i}^n \tau(b_j)
		= \rdt{i-1} - \exc(u_n,\rdt{n}).
	\end{equation*}
\end{proof}

Based on this reformulation, we obtain the following binary program~\eqref{mip:pathbased}.
Recall that we denote the set of all minimal too-long paths $\tlpsymb$ and for a dispatch-node arc $b=(v,k)(v',k')$, we write $a(b)=vv'$ for its underlying arc in $D$.
For a walk $\walksym$ we write $n_\walksym$ for its number of arcs.
The binary variable $y_{at}$ indicates whether $t$ is included in $\calT_a$, and $x_{\walksym it}$ indicates whether $t$ is chosen as the time for the $i$th arc of $\walksym$.
By restricting $t$ to $I_\walksym(i)$, we guarantee that both \cref{lemma:relaxedtoolongtimepoints:start,lemma:relaxedtoolongtimepoints:end} hold, so these conditions are not explicitly imposed.
For an arc $a \in \A{D}$, let $I_a$ denote the union of all intervals $I_\walksym(i)$ where $a(b_i) = a$. 
\begin{subequations}
	\label{mip:pathbased}
	\begin{alignat}{20}
		&\min & \sum_{a\in\A{D}}\sum_{t\in I_a} y_{at} \label{mip:path:objective}\\
		&\mathrm{s.t.}
		& \sum_{t\in I_\walksym(i)} x_{\walksym it} & =&\;& 1
		&&  \forall\walksym\in\tlpsymb,\ i\in(n_\walksym] \label{mip:path:choose}\\
		&& x_{\walksym it} &\leq && y_{a(b_i)t}
		&&  \forall\walksym\in\tlpsymb,\ i\in(n_\walksym],\ t\in I_\walksym(i) \label{mip:path:include}\\
		&& \sum_{t\in I_\walksym(i)} t x_{\walksym it}+\tau(b_i)&
		\geq && \sum_{t\in I_\walksym(i+1)} t x_{\walksym,i+1,t}
		&&  \forall\walksym\in\tlpsymb,\ i\in(n_\walksym) \label{mip:path:link}\\
		&& x_{\walksym it}& \in&& \{0,1\}
		&&  \forall\walksym\in\tlpsymb,\ i\in(n_\walksym],\ t\in I_\walksym(i) \\
		&& y_{at}& \in &&\{0,1\}
		&\quad&  \forall a\in\A{D},\ t\in I_a
	\end{alignat}
\end{subequations}

\begin{remark}
	\textcite{shu_new_2026} also proposed a mixed-integer program for this problem (using their slightly stronger elimination criterion).
	Due to their differing definition of too-long paths that may start only at origin nodes $(o_k,k)$, the MIP is significantly more complex and not reproduced here.
\end{remark}

Alternatively, a very simple heuristic solution for ORP can be obtained by simply adding time points for each too-long path according to its real arrival time:
\begin{theorem}[{{\cite[Theorem~3]{shu_new_2026}}}]\label{theorem:real_arrival_time_eliminates}
	Let $\walksym = b_1\ldots b_n$ denote a too-long path in $D_\solsym$.
	For a discretisation $\calT$ with $\rdt{i-1} \in \calT_{b_i}$ for $i \in (n]$, $\walksym$ is $\calT$-relaxed too-long (and thus eliminated).
\end{theorem}
\textcite[][Section 7.4]{shu_new_2026} also propose a matheuristic for ORP.
They keep all time points from the previous discretisation and then process the too-long paths in batches.
For each batch, they solve their MIP to determine the fewest time points to add to make all paths in the current batch relaxed too-long.
They further restrict the MIP to selecting time points from among those implied by~\cref{theorem:real_arrival_time_eliminates}.

One drawback of these approaches is that they require a full enumeration of all minimal too-long paths, which may be impractical.
\textcite[][Section 4.4]{helber_arc-based_2026} gives an example where a cycle in the dispatch-node graph leads to an exponential number of too-long paths.
Since such cycles are easy to detect and may be handled separately \parencite[see e.g.,][Section 4.1.3.]{marshall_interval-based_2021}, we give another motivating example that does not rely on such cycles.
\begin{example}
	Consider an instance on the network depicted in~\cref{example:pathological}, which consists of the path $\spathsym=v_0\ldots v_n$ together with the nodes $v'$ and $v''$ as well as arcs $v_0v'$, $v'v_1$, $v_{n-1}v''$, and $v''v_{n}$.
	Arcs $v_0v_1$ and $v_{n-1}v_n$ have length~3 and all others have length~1.
	Commodities $k_1$ and $k_2$ must both be transported from $v_0$ to $v_n$.
	Note that the nodes $v'$ and $v''$ exist simply to make $\edt{v}$ and $\ldt{v}$ irrelevant for nodes $v_1$ to $v_{n-1}$.
	\begin{figure}[tbp]
		\centering
		\begin{tikzpicture}
			\node[draw,label={[align=left] above:{$r_{k_1}= 0$\\$r_{k_2}=1$}}] (v0) {$v_0$}; 
			\node[draw, right=3cm of v0] (v1) {$v_1$}; 
			\node[draw, below right=of v0] (vp) {$v'$};
			\node[draw, right=of v1] (v2) {$v_{2}$};
			\node[draw, right=of v2] (vn) {$v_{n-1}$};
			\node[draw, below right=of vn] (vup) {$v''$};
			\node[draw,label={[align=left]above:{$\ell_{k_1}=n+4$\\$\ell_{k_2}=n+5$}}, right=3cm of vn] (vend) {$v_n$};
			
			\draw[->] (v0) to  node[above] {$3$} (v1);
			\draw[->] (v1) to  node[above] {$1$} (v2);
			\draw[->,white,text=black] (v2) to  node {$\dots$} (vn);
			\draw[->] (vn) to  node[below left] {$1$} (vup);
			\draw[->] (v0) to  node[below left] {$1$} (vp);
			\draw[->] (vn) to  node[above] {$3$} (vend);
			\draw[->] (vup) to node[below right] {$1$} (vend);
			\draw[->] (vp) to node[below right] {$1$} (v1);
\end{tikzpicture}
		\caption{Example of an instance with exponentially many too-long paths.
			Arc weights represent travel time $\tau$.}
		\label{example:pathological}
	\end{figure}
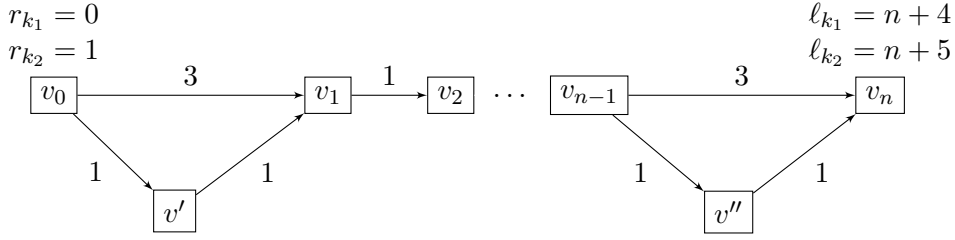
	
	The path $\spathsym$ has length $n + 4$, so each commodity could traverse this path by itself and arrive on time, but when consolidated, $k_1$ arrives too late.
	If we consider the dispatch-node graph for the solution that consolidates them on all arcs of $\spathsym$, we see that the too-long paths must start at $\vdn{v_0}{k_2}$ and end at $\vdn{v_n}{k_1}$, since this is the only way to achieve an arrival time of $n+5$ and the only way to reach the only deadline this violates.
	However, the nodes in between are irrelevant, so we can choose either $\vdn{v_i}{k_1}$ or $\vdn{v_i}{k_2}$ for all intermediate $i\in(n)$.
	Thus, we obtain $2^{n-1}$ minimal too-long paths in this example.

	Despite this abundance of minimal too-long paths, it suffices to add the time points $1$ on arc $v_0v_1$ and $i+2$ on the arcs $v_{i-1}v_{i}$ for $i \in [2,n]$ to eliminate them all (by making them relaxed too-long).
\end{example}
We also see in this example that a necessary ingredient of this blow-up in the number of too-long paths is that on most arcs we can choose between the two commodities.
One can imagine that for instances where good solutions require consolidations of many commodities, the number of too-long paths may grow drastically.
Such instances have already been observed to be harder to solve \parencite[see e.g., the HC classification of][]{marshall_interval-based_2021}.
This motivates us to find a more compact representation of the conflicts encoded by too-long paths.

\subsection{A Graph-Based Criterion}\label{subsec:graph-criterion}
As we just saw, there may be exponentially many too-long paths that can be eliminated in the same way.
To deal with such cases, we want to find a more compact representation that lets us make all minimal too-long paths $\calT$-relaxed too-long, without checking them all individually.

To motivate our upcoming construction, we use the  partial network shown in \cref{fig:motivation-DW} and its dispatch-node graph.
The time bounds not shown in the figure are chosen so that no single-arc path is too-long; their exact values are irrelevant.
For ease of notation, we write $u_{ij}$ for $\vdn{v_i}{k_j}$.
Here, the minimal too-long paths are the paths that start at $u_{01}$ or $u_{02}$ and end at $u_{21}$ or $u_{22}$, except the two paths $u_{01}u_{11}u_{22}$ and $u_{01}u_{12}u_{22}$.
We use the full time discretisation $\calT$, which clearly ensures that all too-long paths are $\calT$-relaxed too-long.
\begin{figure}[tbp]
	\centering
	\begin{tikzpicture}
		\node[draw,label={[align=left] left:{$\edt[k_1]{v_0}= 1$\\$\edt[k_2]{v_0}=2$}}] (v0) {$v_0$};
		\node[draw, right= 2cm of v0] (v1) {$v_1$};
		\node[draw, right=2cm of v1,label={[align=left] right:$\ldt[k_1]{v_2}= 6$\\$\ldt[k_2]{v_2}=7$}] (v2) {$v_2$};
		
		\draw[->] (v0) to  node[above] {\small$\Set{k_1,k_2}$} node[below] {\small$3$} (v1);
		\draw[->] (v1) to  node[above] {\small$\Set{k_1,k_2}$} node[below] {\small$3$} (v2);
	\end{tikzpicture}
	
	\medskip
	
	\begin{tikzpicture}
		\node[draw] (v01) {$v_0,k_1$};
		\node[draw, below=of v01] (v02) {$v_0,k_2$};
		\node[draw, right=of v01] (v11) {$v_1,k_1$};
		\node[draw, right=of v02] (v12) {$v_1,k_2$};
		\node[draw, right=of v11] (v21) {$v_2,k_1$};
		\node[draw, right=of v12] (v22) {$v_2,k_2$};
		
		\draw[->] (v01) to (v11);
		\draw[->] (v01) to (v12);
		\draw[->] (v02) to (v11);
		\draw[->] (v02) to (v12);
		\draw[->] (v11) to (v21);
		\draw[->] (v11) to (v22);
		\draw[->] (v12) to (v21);
		\draw[->] (v12) to (v22);
	\end{tikzpicture}
	\caption{Illustrative example to motivate the construction of $D_\tlpsymb$.}
	\label{fig:motivation-DW}
\end{figure}
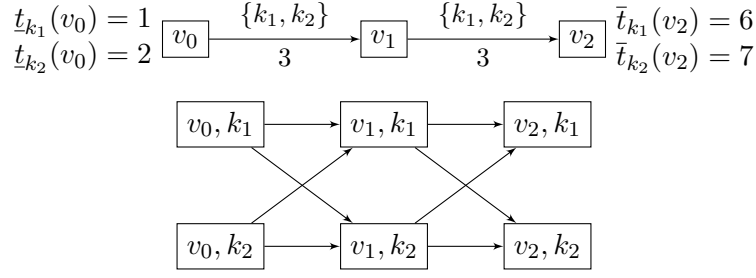

First, we observe that certain too-long paths are \enquote{more critical} than others.
For example, consider the two too-long paths $\walksym=u_{01}u_{11}u_{21}$ and $\walksym'=u_{02}u_{11}u_{21}$. Then $\rat{2} = 7$ and $\rat[\walksym']{2} = 8$.
Thus, it suffices to make $\walksym$ relaxed too-long, since $\walksym'$ is then relaxed too-long as well.
As a result, we are interested in the \enquote{most critical arrival times} at the nodes where minimal too-long paths end.

In order to compute them, the following propagation approach would be natural:
since the most critical arrival time at $u$ is obtained by some minimal too-long path $\walksym=u_0\ldots u_{n-1}u_n$ with $u_n=u$, its most critical arrival time $\rat{n}$ is equal to $\rounddown{\rat{n-1}}{u_{n-1}u_n} + \tau(u_{n-1}u_n)$.
Thus, we could simply compute these values for the predecessors of $u$ (on minimal too-long paths) and take the smallest one.

This propagation does not work, however.
In our small example, we would obtain that $u_{22}$ has $7$ as its most critical arrival time, which does not exceed $\ldt[]{u_{22}}$ despite the fact that all minimal too-long paths are eliminated.
This occurs since there is \enquote{cross-contamination}:
the minimal too-long path $\walksym=u_{01}u_{12}u_{21}$ propagates its relaxed arrival time of $4$ to $u_{12}$, which is then moved to $u_{22}$ to obtain the $7$ using the walk $\walksym'=u_{02}u_{12}u_{22}$.
But the walk $u_{01}u_{12}u_{22}$ that actually leads to the $7$ is not even too-long.

Thus, we need to ensure that this cross-contamination does not occur.
In the example above, this happened since $\rat{1} = 4$ but $\rat[\walksym']{1}=5$, so the walk obtained by retracing the minimal one was not long enough.
We fix this by means of a time-expansion:
we augment the nodes $u_{ij}$ by their real arrival times in the minimal too-long paths they appear in.
This ensures that all prefixes that meet at a node have equal real arrival time and can be used interchangeably to obtain too-long paths.
Thus, only the worst (earliest) one is of interest and this motivates the following definition.
\begin{definition}[Time-expanded too-long path graph]
	The \emph{time-expanded too-long path graph $D_\tlpsymb$} is given by
	\begin{align*}
		\V{D_\tlpsymb} &= \Set{(u_i, \rdt{i}) : \walksym=u_0\ldots u_n\in\tlpsymb, i\in [n]}, \\
		\A{D_\tlpsymb} &= \Set{(u_i,\rdt{i})(u_{i+1},\rdt{i+1}) : \walksym=u_0\ldots u_n\in\tlpsymb, i\in [n)}.
	\end{align*}
	For notational purposes, we identify a node $w=(u,t)$ with its first component $u$, letting us write $\edt[]{w}$ or $a(w)$.

	We also define $\calS_\tlpsymb$ and $\calB_\tlpsymb$ to be the sets of nodes in $D_\tlpsymb$ originating from the initial and final nodes of walks in $\tlpsymb$.
\end{definition}

This graph expands the minimal too-long paths by their real arrival times.
Since we frequently move between a minimal too-long path and its time-expanded version, we introduce some helpful notation.
\begin{notationenum}
	\begin{itemize}
		\item For $\walksym=u_0\ldots u_n\in\tlpsymb$, we write $\spathsym_\walksym$ for the path $(u_0,\rdt{0})\ldots(u_n,\rdt{n})$ in $D_\tlpsymb$.
		\item For a path $\spathsym=w_0\ldots w_n$ in $D_\tlpsymb$ with $w_i=(u_i,t_i)$ for $i\in[n]$, we write $\walksym_\spathsym$ for the walk $u_0\ldots u_n$ in $D_\solsym$.
	\end{itemize}
\end{notationenum}
Note that, using this notation, $D_\tlpsymb = \bigcup_{\walksym\in\tlpsymb} \spathsym_\walksym$.

The rest of this subsection is essentially dedicated to showing that, on this graph, the propagation described above actually works.
To this end, we first note a few properties.
\begin{observation}
	\label{SW-BW-characterisation}
	Let $w=(u,t)\in\V{D_\tlpsymb}$.
	\begin{enumerate}
		\item $\Nin{w} = \emptyset$ if and only if $w\in\calS_\tlpsymb$ if and only if $\delay(w) = 0$.
		\item $\Nout{w} = \emptyset$ if and only if $w\in\calB_\tlpsymb$ if and only if $\exc(w)>0$.
	\end{enumerate}
\end{observation}
\begin{proof}
	By \cref{used-buffer-lateness-minimal-tlp}, the nodes $w\in D_\tlpsymb$ that have $\delay(w)=0$ are exactly the initial nodes of minimal too-long paths.
	Similarly, the nodes with $\exc(w)>0$ are exactly the final nodes of minimal too-long paths.
\end{proof}

\begin{corollaryenum}
	\label{tlp-properties-in-tetlp}
	\begin{enumprop}
		\item If $\spathsym$ is an $\calS_\tlpsymb\calB_\tlpsymb$-path in $D_\tlpsymb$, then $\walksym_\spathsym$ is too-long.
		\item If $\spathsym$ is a path in $D_\tlpsymb$ and $\walksym_\spathsym$ is too-long, then $\spathsym$ ends at a node in $\calB_\tlpsymb$.
	\end{enumprop}
\end{corollaryenum}
\begin{proof}
	We notice that each arc $(u,t)(u',t')$ in $D_\tlpsymb$ satisfies $t' = t + \tau(uu')$, giving us that any path $\spathsym$ from $w=(u,t)$ to $w'=(u',t')$ satisfies $t' = t + \tau(\walksym_\spathsym)$.
	Now let $\spathsym=w_0\ldots w_n$ be a path in $D_\tlpsymb$ with $w_i=(u_i,t_i)$ for all $i\in[n]$.
	
	For the first part, let $\spathsym$ be an $\calS_\tlpsymb\calB_\tlpsymb$-path. Then $\edt[]{u_0} + \tau(\walksym_\spathsym) = t_0 + \tau(\walksym_\spathsym) = t_n > \ldt[]{u_n}$, by \cref{SW-BW-characterisation}.
	Thus, $\walksym_\spathsym$ is too-long.
	
	For the second part, $\ldt[]{w_n} < \edt[]{w_0} + \tau(\walksym_\spathsym)$ since this path is too-long.
	Using \cref{SW-BW-characterisation}, $\edt[]{w_0} \leq t_0$, giving us $\ldt[]{w_n} < t_0 + \tau(\walksym_\spathsym) = t_n$, so $w_n\in\calB_\tlpsymb$.
\end{proof}

Next, let us define the propagation we already motivated at the start of the subsection.
\begin{definition}[Relaxed arrival times in {{$\exgraph$}}]
	For $w\in D_\tlpsymb$, we define the \emph{relaxed arrival time $\ratw{w}$ at $w$} by
	\begin{displaymath}
		\ratw{w} = \begin{cases}
			\edt[]{w} & \text{if } \Nin{w}=\emptyset, \\
			\min\Set{\rounddown{\ratw{w'}}{w'w} + \tau(w'w) : w'\in\Nin{w}} & \text{otherwise}.
		\end{cases}
	\end{displaymath}
\end{definition}
Note that this is well defined since $D_\tlpsymb$ is acyclic and  $\Nin{w} = \emptyset$ if and only if $t=\edt[]{w}$.

We now establish the desired characterisation, namely that making the critical prefixes computed by the propagation relaxed too-long is necessary and sufficient.
\begin{theorem}
	\label{criterion:tetlpg}
	Let $\calT$ be a time discretisation.
	Then all minimal too-long paths in $D_\solsym$ are $\calT$-relaxed too-long if and only if $\ratw{w}  > \ldt[]{w}$ for all $w\in\calB_\tlpsymb$.
\end{theorem}	
\begin{proof}
	We start with the backward implication for which we need the following claim.
	\begin{claim}
		For all $w=(u,t)\in\V{D_\tlpsymb}$ and for all $\walksym\in\tlpsymb$ for which $u$ is the $i$th node and $t=\rdt{i}$, we get $\ratw{w} \leq \rat{i}$.
	\end{claim}
	\begin{claimproof}
		We proceed by induction on a topological ordering. For nodes $w=(u,t)\in\calS_\tlpsymb$, if $u$ occurs in $\walksym\in\tlpsymb$, then $u$ is the initial node and $\ratw{w} = \edt[]{u} =\rat{0}$.
		
		For the induction step, we let $w=(u,t)\in\V{D_\tlpsymb}\setminus\calS_\tlpsymb$, $\walksym\in\tlpsymb$ contain $u$ as its $i$th node, and $t=\rdt{i}$.
		Let $u'$ be the predecessor of $u$ on $\walksym$ and $t' = \rdt{i-1}$.
		Then $w'=(u',t')$ is a predecessor of $w$ in $D_\tlpsymb$.
		Thus, using the induction hypothesis, we conclude that $\ratw{w} \leq \rounddown{\ratw{w'}}{w'w} + \tau(w'w) \leq \rounddown{\rat{i-1}}{u'u} + \tau(u'u) = \rat{i}$.
	\end{claimproof}
	
	Now let $\walksym=u_0\ldots u_n\in\tlpsymb$.
	We know that $w=(u_n,\rdt{n})\in\calB_\tlpsymb$ and $\ldt[]{w} < \ratw{w} \leq \rat{n}$, so $\walksym$ is $\calT$-relaxed too-long.
	
	For the forward implication, we need to associate the nodes of $D_\tlpsymb$ with walks in $D_\solsym$ that realise their relaxed arrival times.
	\begin{claim}
		\label{path-realising-ratw}
		For each $w\in\V{D_\tlpsymb}$ there exists a path $\spathsym=w_0\ldots w_n$ in $D_\tlpsymb$ such that $w_n = w$, $w_0\in\calS_\tlpsymb$, and, for all $i\in[n]$,
		\begin{enumprop}
			\item $\rdt[\walksym_\spathsym]{i} = t_i$ and 
			\item $\rat[\walksym_\spathsym]{i} = \ratw{w_i}$.
		\end{enumprop}
	\end{claim}
	\begin{claimproof}
		We prove the claim by induction on some topological ordering, starting with $w_0=(u_0,t_0)\in\calS_\tlpsymb$ once more.
		In this case, we can take the trivial path $\spathsym=w_0$ and get $\rdt[\walksym_\spathsym]{0} = \edt[]{u_0} = t_0$ and $\ratw{w_0} = \edt[]{u_0} = \rat[\walksym_\spathsym]{0}$.
		
		For the induction step, let $w=(u,t)\in\V{D_\tlpsymb}\setminus\calS_\tlpsymb$.
		Then $\ratw{w} = \rounddown{\ratw{w'}}{w'w} + \tau(w'w)$ for some $w'=(u',t')\in\Nin{w}$.
		By definition of the arcs, $t' = t - \tau(w'w)$.
		
		Then, by induction, we get a path $\spathsym'=w_0\ldots w_{n'}$ in $D_\tlpsymb$ with $w_0\in\calS_\tlpsymb$, $w_{n'}=w'$, and, for all $i\in[n']$, $w_i=(u_i,t_i)$, $\rdt[\walksym_{\spathsym'}]{i} = t_i$ and $\ratw{w_i} = \rat[\walksym_{\spathsym'}]{i}$.
		Let $\spathsym=w_0\ldots w_{n'}w$ and $n=n'+1$. Then $w_0\in\calS_\tlpsymb$, $w_n=w$ and the remaining two properties continue to hold for $i\in[n']$.
		For $n$ we get $\rdt[\walksym_\spathsym]{n} = \rdt[\walksym_\spathsym]{n'} + \tau(w'w) = t = t_n$ and $\ratw{w} = \rounddown{\ratw{w'}}{w'w} + \tau(w'w) = \rounddown{\rat[\walksym_\spathsym]{n'}}{u'u} + \tau(u'u) = \rat[\walksym_\spathsym]{n}$, as desired.
	\end{claimproof}
	
	We can also show a lower bound on the values $\ratw{w}$.
	\begin{claim}
		\label{ratw-at-least-edt}
		For all $w\in \V{D_\tlpsymb}$, $\ratw{w} \geq \edt[]{w}$.
	\end{claim}
	\begin{claimproof}
		We prove the result inductively, where it clearly holds for nodes in $\calS_\tlpsymb$.
		So, let $w=(u,t)\in\V{D_\tlpsymb}\setminus\calS_\tlpsymb$.
		By definition of $D_\tlpsymb$, $w$ is on $\spathsym=\spathsym_\walksym$ for some minimal too-long path $\walksym\in\tlpsymb$.
		By \cref{path-realising-ratw} we obtain a $w_0w$-path $\spathsym'$ in $D_\tlpsymb$ and corresponding walk $\walksym'=\walksym_{\spathsym'}$ with $w_0\in\calS_\tlpsymb$ that realises the relaxed arrival times.
		
		We now regard the splicings $\spathsym'' = \spathsym'w\spathsym = w_0\ldots w_n$ and $\walksym'' = \walksym'u\walksym = u_0\ldots u_n$.
		Since $\spathsym''$ starts in $\calS_\tlpsymb$ and ends in $\calB_\tlpsymb$, $\walksym''$ is too-long by \cref{tlp-properties-in-tetlp}.
		It thus contains a minimal too-long path $\walksym^*=u_i\walksym''u_j$, meaning $\spathsym^*=w_i\spathsym''w_j$ ends in $\calB_\tlpsymb$ by \cref{tlp-properties-in-tetlp}.
		Hence, $j=n$.
		
		Since $\walksym$ is a minimal too-long path and $w\notin\calS_\tlpsymb$, $u\walksym$ is not too-long, and, thus, $u=u_s$ is in $\walksym^*$.
		Consequently, $i<s$, $u_i\in \walksym'$, and $\ratw{w_i} = \rat[\walksym']{i}$.
		By induction, we know that $\rat[\walksym']{i} = \ratw{w_i}\geq \edt[]{w_i} = \rat[\walksym^*]{0}$.
		Therefore, $\rat[\walksym']{i+\ell} \geq \rat[\walksym^*]{\ell}$ for all $\ell\leq s-i$, in particular, $\ratw{w} = \rat[\walksym']{s} \geq \rat[\walksym^*]{s-i}$ holds at $w$.
		By \cref{minimal-tlp-and-rtl-properties}, $\rat[\walksym^*]{s-i} > \edt[]{w}$, so $\ratw{w} > \edt[]{w}$ and the claim follows.
	\end{claimproof}
	
	Now let $w\in\calB_\tlpsymb$.
	\Cref{path-realising-ratw} yields a path $\spathsym=w_0\ldots w_n$ in $D_\tlpsymb$ with $w_i=(u_i,t_i)$ for all $i\in[n]$ and $w_0\in\calS_\tlpsymb$ and $w_n=w$.
	Again, we obtain that $\walksym_\spathsym$ is a too-long path and contains a minimal too-long path $\walksym^* = u_i\walksym_\spathsym u_n$, by \cref{tlp-properties-in-tetlp}.

	Since $\walksym^*$ is $\calT$-relaxed too-long, $\rat[\walksym^*]{n-i} > \ldt[]{u_n} = \ldt[]{w}$.
	Moreover, by \cref{ratw-at-least-edt}, $\rat[\walksym_\spathsym]{i} = \ratw{w_i} \geq \edt[]{w_i} = \rat[\walksym^*]{0}$ and, inductively, $\rat[\walksym_\spathsym]{i+\ell} \geq \rat[\walksym^*]{\ell}$.
	Thus, $\ratw{w} = \rat[\walksym_\spathsym]{n} \geq \rat[\walksym^*]{n-i} > \ldt[]{w}$ as desired.
\end{proof}

\begin{remark}
	We note that we only needed to use the set $\tlpsymb$ of all minimal too-long paths in the proof above to get the forward implication.
	For the converse, any set $\tlpsymb'$ of minimal too-long paths would have worked, so, indeed, if $\ratw{w} > \ldt[]{w}$ for all $w\in\boundaries_{\tlpsymb'}$, then all minimal too-long paths in $\tlpsymb'$ are $\calT$-relaxed too-long.
\end{remark}

\subsection{Constructing the Time-Expanded Too-Long Path Graph}\label{subsec:constructing-tetlpg}
In order to use the criterion obtained in \cref{criterion:tetlpg}, we need to be able to construct the time-expanded too-long path graph $\exgraph$.
Evidently, it can be obtained by first enumerating all minimal too-long paths in $\tlpsymb$.
However, we want to avoid this explicit enumeration and this section describes an algorithm that constructs $\exgraph$ in $\bigO(\abs{\A{D_\solsym}} \cdot H)$ time.
While this is still pseudo-polynomial in the instance parameters, it avoids the enumeration of potentially exponentially many minimal too-long paths.

To do so, we first construct a \emph{candidate graph} $\candgraph$ that contains $\exgraph$ as a subgraph and then remove the unnecessary nodes and arcs.
The algorithm to compute the candidate graph is given in \cref{algo:candidate}.
It starts by adding timed copies of all nodes at their earliest time.
It then expands each timed copy along outgoing arcs, so long as the expansion still has positive slack consumption (otherwise, this timed arc cannot be induced by a minimal too-long path).
The nodes reached this way are also considered for expansion if they have non-positive lateness (otherwise, they are already too late).
We now establish the running-time bound and prove that $\exgraph$ is indeed a subgraph.
\begin{algorithm}[tbp]
	\caption{Building the candidate graph}
	\label{algo:candidate}
	\SetKwFunction{constructCandidateGraph}{constructCandidateGraph}
	\Def{\constructCandidateGraph{$D_\solsym$}}{
		$\candgraph \gets (\Set{(u,\edt[]{u}) : u \in \V{D_\solsym}},\emptyset)$\;\label{algo:candidate:basecopies}
		$\labels\gets \V{\candgraph}$\;
\While{$\labels \neq \emptyset$}{
			pop $w=(u,t)$ from $\labels$ with minimal $t$\;
			\For{$u' \in \Nout{u}$}{
				$t' \gets t + \tau({uu'})$, $w'\gets (u',t')$ \;
				\uIf{$\delay(w') > 0$} {
					insert $w'$ into $\V{\candgraph}$ \;
					insert $ww'$ into $\A{\candgraph}$ \;
					\uIf{$\exc(w')\leq 0$}{
						insert $w'$ into $\labels$\;
					}
				}
			}
		}
		\Return{$\candgraph$}\;
	}
\end{algorithm}
\begin{lemma}[Candidate graph]\label{candgraph-contains-exgraph}
	\Cref{algo:candidate} runs in $\bigO(\abs{\A{D_\solsym}}\cdot H)$ time and the computed \emph{candidate graph} $\candgraph$ satisfies $\exgraph \subseteq \candgraph$.
\end{lemma}
\begin{proof}
	For the time complexity, note that $\labels$ can only contain nodes $(u,t)$ with $0\leq \edt[]{u} \leq t \leq \ldt[]{u} < H$.
	Since we process the nodes with increasing $t$, no node is processed twice and the iterations of the inner for loop are thus bounded by $\abs{\A{D_\solsym}}\cdot H$.
	Extracting the node $(u,t)$ from $\labels$ can be done in constant time by using a separate list for each of the $H$ lengths in question.
	All other operations are easy to do in constant time, yielding the desired running time.
	
	To see that $\exgraph\subseteq\candgraph$, let $\walksym=u_0\ldots u_n\in\tlpsymb$ and let $\spathsym_\walksym = w_0\ldots w_n$.
	We show by induction that $\spathsym_\walksym\subseteq \candgraph$, which also shows $\exgraph\subseteq\candgraph$ since $\walksym\in\tlpsymb$ was arbitrary.
	
	By~\cref{algo:candidate:basecopies}, $w_0 \in \candgraph$.
	Now let $w_i\in \candgraph$ for $i<n$. Then $\exc(w_i)\leq 0$ by \cref{used-buffer-lateness-minimal-tlp}.
	Thus, $w_i$ was added to $\labels$ in some iteration of \cref{algo:candidate} and, when it is removed, $u_{i+1} \in \Nout{u_i}$.
	Since $\delay(w_{i+1}) > 0$ by \cref{used-buffer-lateness-minimal-tlp}, the next node $w_{i+1}$ as well as the arc $w_iw_{i+1}$ are in $\candgraph$, too.
\end{proof}

Note that, by construction, the nodes without incoming arcs are exactly the nodes in $\starts' = \Set{w\in \V{\candgraph} : \delay(w) = 0}$.
Similarly, a superset of $\boundaries_\tlpsymb$ is given by $\boundaries' = \Set{w \in \V{\candgraph} : \exc(w) > 0}$.
Note that in contrast to $\exgraph$, there may be nodes $w \in \candgraph$ with $\exc(w) \leq 0$ that have no outgoing arcs.

We still need to determine which nodes and arcs are actually needed.
We focus on the arcs, since if we can correctly determine these, then we can discard all isolated nodes afterwards (assuming that $\edt[]{u}\leq\ldt[]{u}$ for all $u\in D_\solsym$).

Note that an arc $ww'\in \candgraph$ is part of $\exgraph$ if and only if $ww'$ is on a path $\spathsym_\walksym$ for some $\walksym\in\tlpsymb$.
This path is also part of $\candgraph$ and satisfies the properties listed in \cref{used-buffer-lateness-minimal-tlp-equiv}.
In particular, it starts in $\starts'$ and ends in $\boundaries'$.
Since the lemma is a characterisation, we can also get the converse:
\begin{corollary}
	\label{arcs-in-exgraph}
	An arc $ww'\in \candgraph$ is in $\exgraph$ if and only if there exists a $\starts'\boundaries'$-path $\spathsym=w_0\ldots w_n$ in $\candgraph$ that contains $ww'$ and for which $\delay(w_i) \geq \exc(w_n)$ holds for all $i\in(n]$.
\end{corollary}
\begin{proof}
	The forward implication follows from the preceding argument.
	For the converse, let $\spathsym=w_0\ldots w_n$ be a path in $\candgraph$ that contains $ww'$ and for which $\delay(w_i) \geq \exc(w_n)$ holds for all $i\in(n]$.
	Then $\walksym_\spathsym$ is a walk in $D_\solsym$ for which the properties of \cref{used-buffer-lateness-minimal-tlp-equiv} are satisfied:
	the first is guaranteed by our construction already since we do not add outgoing arcs to nodes with positive lateness and $\exc(w_n)>0$ holds since $w_n\in\boundaries'$.
\end{proof}

We therefore need an efficient way of verifying whether an arc $ww'\in\candgraph$ is on such a path.
This means we need an $\starts'w$-path $\spathsym_\ell$ and a $w'\boundaries'$-path $\spathsym_r$ that, when combined, satisfy the condition of the corollary.
But this means that, for the prefix $\spathsym_\ell$, a larger slack consumption is always better.
Thus, we say that $w$ is \emph{$s$-reachable} if there exists an $\starts'w$-path in $\candgraph$ whose nodes, except the first, all have slack consumption at least $s$.
Thus, we want to compute the best such $s$, which we call
\begin{displaymath}
	\capacity(w) = \sup\Set{s : w \text{ is $s$-reachable}}.
\end{displaymath}
In particular, $\capacity(w)=\infty$ for all $w\in\starts'$ since they are $s$-reachable for any value of $s$.

For the suffix $\spathsym_r$, our goal is to be as unrestrictive as possible.
Thus, lower values of $\exc(w_n)$ for the end node $w_n\in\boundaries'$ are desirable, as long as we can reach them by a path whose minimum slack consumption is at least this value.
Similarly to the above, we call $w'$ \emph{$\ell$-extendable} if there exists a $w'\boundaries'$-path in $\candgraph$ ending at $w_n\in\boundaries'$ such that the slack consumption of all its nodes is at least $\exc(w_n)=\ell$.
We define
\begin{displaymath}
	\requirement(w) = \min \Set{\ell : w \text{ is $\ell$-extendable}}.
\end{displaymath}
Here and below, we use the convention $\min\emptyset=+\infty$.
In particular, $\requirement(w)=\infty$ if $w\in\starts'$.

These two properties let us characterise which arcs of $\candgraph$ are actually part of $\exgraph$.
\begin{theorem}\label{marking-rule} 
	An arc $ww' \in \A{\candgraph}$ is in $\exgraph$ if and only if
	\begin{equation}\label{marking-rule:eq}
		\requirement(w') < \infty \qquad \text{and} \qquad \capacity(w) \geq \requirement(w').
	\end{equation}
\end{theorem}
\begin{proof}
	Similarly, if $ww'\in \A{\candgraph}$ is in $\exgraph$, then \cref{arcs-in-exgraph} yields a path $\spathsym=w_0\ldots w_n$ in $\candgraph$ that contains $ww'$ and for which $\delay(w_i) \geq \exc(w_n)$ holds for all $i\in(n]$.
	In particular, $w$ is $\exc(w_n)$-reachable and $w'$ is $\exc(w_n)$-extendable, yielding $\capacity(w) \geq \requirement(w')$ and $\requirement(w') < \infty$.
	
	Conversely, assume \cref{marking-rule:eq} holds for the arc $ww'$.
	Let $\spathsym_\ell=w_0\ldots w_k$, with $w_k=w$, be a path realising $\capacity(w)$ and $\spathsym_r=w_{k+1}\ldots w_n$, with $w_{k+1}=w'$, be a path realising $\requirement(w')$, which exists since the value is finite by assumption.
	Then $\delay(w_i) \geq \exc(w_n) = \requirement(w')$ for all $i\geq k+1$ and $\delay(w_i) \geq \capacity(w) \geq \requirement(w') = \exc(w_n)$ for $0<i\leq k$.
	Hence, the path $\spathsym_{\ell}\spathsym_r$ is a $\starts'\boundaries'$-path that contains $ww'$ and satisfies the condition of \cref{arcs-in-exgraph}, so $ww'\in\exgraph$.
\end{proof}

Let us illustrate this on a small example.
\begin{example}
	Consider the dispatch-node graph $D_\solsym$ in~\cref{fig:markingrule:dispatchgraph} and the corresponding candidate graph $\candgraph$ in~\cref{fig:markingrule:candidategraph}.
	The minimal too-long paths are $u_1u_2$ and $u_0u_1u_3$:
	the path $u_1u_3$ arrives at time $4$ and is feasible, whereas $u_0u_1u_3$ arrives at time $5 > \ldt[]{u_3}$.
	The check of \cref{marking-rule} keeps their time-expanded arcs and rejects the arc from $(u_1,2)$ to $(u_2,5)$.
	This also lets us remove the now isolated nodes $(u_2,5)$, $(u_3,2)$, and $(u_2,2)$.
	
	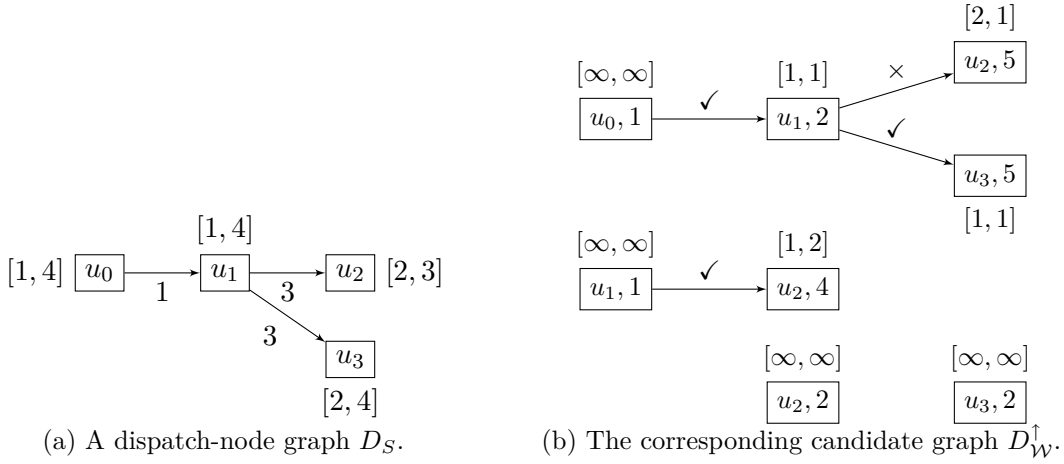
\begin{figure}[tbp]
		\centering
		\begin{subfigure}[b]{0.48\textwidth}
			\centering
			\begin{tikzpicture}
				\node[draw,label={[align=left]left:{$[1,4]$}}] (u0) {$u_0$};
				\node[draw,right=of u0, label={[align=left]above:{$[1,4]$}}] (u1) {$u_1$};
				\node[draw, right=of u1,label={[align=left]right:{$[2,3]$}}] (u2) {$u_2$};
				\node[draw, below=0.65cm of u2,
				label={[align=left]below:{$[2,4]$}}] (u3) {$u_3$};
				
				\draw[->] (u0) to node[below] {$1$} (u1);
				\draw[->] (u1) to node[below] {$3$} (u2);
				\draw[->] (u1) to node[below left] {$3$} (u3);
				
			\end{tikzpicture}
			\caption{A dispatch-node graph $D_\solsym$.}
			\label{fig:markingrule:dispatchgraph}
		\end{subfigure}\hfill
		\begin{subfigure}[b]{0.48\textwidth}
			\centering
			\begin{tikzpicture}[x=1.65cm,y=1.5cm,
				every node/.style={font=\small}]
				\node[draw,label={above:{$[\infty,\infty]$}}] (u0t1) at (0,1) {$u_0,1$};
				\node[draw,label={above:{$[1,1]$}}] (u1t2) at (1.5,1) {$u_1,2$};
				\node[draw,label={above:{$[2,1]$}}] (u2t5) at (3,1.5) {$u_2,5$};
				\node[draw,label={below:{$[1,1]$}}] (u3t5) at (3,0.5) {$u_3,5$};
				\node[draw,label={above:{$[\infty,\infty]$}}] (u1t1) at (0,-0.5) {$u_1,1$};
				\node[draw,label={above:{$[1,2]$}}] (u2t4) at (1.5,-0.5) {$u_2,4$};
				\node[draw,label={above:{$[\infty,\infty]$}}] (u2t2) at (1.5,-1.5) {$u_2,2$};
				\node[draw,label={above:{$[\infty,\infty]$}}] (u3t2) at (3,-1.5) {$u_3,2$};

\draw[->] (u0t1) to node[above] {$\checkmark$} (u1t2);
				\draw[->] (u1t1) to node[above] {$\checkmark$} (u2t4);
				\draw[->] (u1t2) to node[above] {$\times$} (u2t5);
				\draw[->] (u1t2) to node[above] {$\checkmark$} (u3t5);
			\end{tikzpicture}
			\caption{The corresponding candidate graph $\candgraph$.}
			\label{fig:markingrule:candidategraph}
		\end{subfigure}
		
		\caption{Illustration of the marking rule in~\cref{marking-rule}.
			The nodes $u$ in the dispatch-node graph are labelled with $[\edt[]{u},\ldt[]{u}]$.
			In the candidate graph, the nodes $w$ are labelled by $[\requirement(w),\capacity(w)]$.}
	\end{figure}
\end{example}

To complete the algorithm, we want to efficiently determine $\capacity(w)$ and $\requirement(w)$.
\begin{observation}
	\label{cap-req-recurrences}
	The following recurrences determine $\capacity(w)$ and $\requirement(w)$.
	\begin{align}\label{eq:cap-recurrence}
		\capacity(w) & = \begin{cases}
			\infty & w \in \starts' \\
			\min \Set{\delay(w), \max_{w' \in \Nin{w}} \capacity(w')} & w \notin \starts'
		\end{cases}\\\label{eq:req-recurrence}
		\requirement(w) &= \begin{cases}
			\exc(w) & w \in \boundaries' \\
			\min \Set{\requirement(w') : w'\in\Nout{w},\, \delay(w)\geq \requirement(w')} & w\notin\boundaries'
		\end{cases}
	\end{align}
	
	Using these, we can compute the values of $\capacity(w)$ and $\requirement(w)$ for all $w \in \V{\candgraph}$ in $\bigO(\abs{\V{\candgraph}} + \abs{\A{\candgraph}})$ time.
\end{observation}   
\begin{proof}
	We start with the recurrence for $\capacity(w)$ and inductively obtain that the formula is correct, where this is true for all nodes in $\starts'$.
	For the remaining nodes, $\capacity(w)$ is determined by a path to $w$ whose nodes have a minimum slack consumption that is as large as possible.
	Thus, it uses a predecessor $w'$ of maximum $\capacity(w')$, without loss of generality, and $\capacity(w) = \min\Set{\capacity(w'),\delay(w)}$.
	
	For $\requirement(w)$, we again see that the value for $w\in\boundaries'$ is set correctly since $\delay(w)\geq \exc(w)$.
	For the remaining nodes, the path determining $\requirement(w)$ (if it exists) uses a successor $w'$ that minimises $\requirement(w')$ subject to $\delay(w)\geq \requirement(w')$, and its requirement is then identical.
	If $\requirement(w)=\infty$, then any path to a node in $\boundaries'$ has insufficient slack consumption; in particular, $\delay(w) < \requirement(w')$ for all successors $w'$.

	Since for either recurrence each node and arc is processed exactly once, the procedure completes in $\bigO(\abs{\V{\candgraph}} + \abs{\A{\candgraph}})$ time.
\end{proof}

To summarise this section, we obtain \cref{algo:exgraph} that constructs $\exgraph$.
Evidently, computing $\capacity(w)$ and $\requirement(w)$ requires no more operations than determining $\candgraph$ in the first place, so with~\cref{candgraph-contains-exgraph}, \cref{marking-rule}, and \cref{cap-req-recurrences}, \cref{algo:exgraph} computes $\exgraph$ in $\bigO(\abs{\A{D_\solsym}} \cdot H)$ time. \begin{algorithm}[tbp]
	\caption{Computing $\exgraph$.}
	\label{algo:exgraph}
	\SetKwFunction{constructTimeExpandedGraph}{constructTimeExpandedGraph}
	\Def{\constructTimeExpandedGraph{$D_\solsym$}}{
		$\candgraph\gets \constructCandidateGraph{$D_\solsym$}$\;
		Compute $\capacity(w)$ and $\requirement(w)$ for all $w\in\candgraph$ using \eqref{eq:cap-recurrence} and \eqref{eq:req-recurrence}\;
		$\exgraph \gets (\emptyset,\emptyset)$\;
		\For{$ww'\in\A{\candgraph}$}{
			\If{$\requirement(w')<\infty$ \textbf{and} $\capacity(w)\geq \requirement(w')$}{
				Add $w$, $w'$, and $ww'$ to $\exgraph$\;
			}
		}
		\Return{$\exgraph$}\;
	}
\end{algorithm}

We further improve the practical running time of~\cref{algo:exgraph} by preprocessing $D_\solsym$ to remove nodes and arcs that cannot lie on too-long paths.
The idea is to compute for each $u \in \V{D_\solsym}$ the latest real arrival time of any walk in $D_\solsym$, i.e.,
\[
    t_\mathrm{max}(u) = \max \Set{\rdt{n} : \exists \walksym = u_0\ldots u_n,\ u_n=u},
\]
and similarly the earliest real arrival time a walk must have at $u$ to be extendable into a too-long path, i.e.,
\[
    t_\mathrm{min}(u) = \min\Set{\ldt[]{u_n} - \tau(\walksym) : \exists \walksym = u_0\ldots u_n,\ u_0=u}.
\]
Evidently, if $t_\mathrm{max}(u) \leq t_\mathrm{min}(u)$, it cannot lie on a too-long path and can be removed.
If $D_\solsym$ is acyclic, these values are easy to compute by forward and backward propagation.
Otherwise, we can first determine the condensation graph of $D_\solsym$ by merging every strongly connected component into a single node, and then determine these values on the resulting graph. Every node in a strongly connected component containing a cycle is on a too-long path in any case and thus we set $t_\mathrm{min}(u) =- \infty$ and $t_\mathrm{max}(u) = \infty$ for them.

\subsection{Solving the ORP using \texorpdfstring{$\exgraph$}{D_W}}\label{subsec:solving-ORP-graph}

Similarly to \cref{lemma:relaxedtoolongtimepoints}, we now prove another alternative characterisation that lends itself better to the formulation of an IP.
\begin{lemma}\label{characterisation-IP-calW-rtl}
	All minimal too-long paths are $\calT$-relaxed too-long if and only if there exist time points $t(w'w)\in\calT_{w'w}$ for all $w'w\in\A{\exgraph}$ and $\pi(w)$ for all $w\in\V{\exgraph}$ such that
	\begin{enumprop}
		\item $\pi(w) \leq \edt[]{w}$ for all $w\in\starts_\tlpsymb$,
		\item $t(w'w) \leq \pi(w')$ for all $w'w\in\A{\exgraph}$,\label{prop:calW-rtl-2}
		\item $t(w'w) + \tau(w'w) \geq \pi(w)$ for all $w'w\in\A{\exgraph}$, and
		\item $\pi(w) > \ldt[]{w}$ for all $w\in\boundaries_\tlpsymb$.
	\end{enumprop}
	Additionally, for $w'w\in\A{\exgraph}$ with $w'=(u',t')$, $t(w'w)\in I_{w'w} = (t'-\requirement(w),t']$.
\end{lemma}
\begin{proof}
	Note that all minimal too-long paths are $\calT$-relaxed too-long if and only if the relaxed arrival times $\ratw{w}$ at the boundary nodes $w\in\boundaries_\tlpsymb$ satisfy $\ratw{w} > \ldt[]{w}$, by \cref{criterion:tetlpg}.
	
	For the forward implication, we use $\pi(w) = \ratw{w}$ and $t(w'w) = \rounddown{\ratw{w'}}{w'w}$.
	Then, for $w\in\starts_\tlpsymb$, $\pi(w) = \ratw{w} = \edt[]{w}$ and, for $w\in\boundaries_\tlpsymb$, $\pi(w) = \ratw{w} > \ldt[]{w}$.
	For $w'w\in\A{\exgraph}$, $t(w'w) = \rounddown{\ratw{w'}}{w'w} \leq \ratw{w'} = \pi(w')$ and $t(w'w) + \tau(w'w) = \rounddown{\ratw{w'}}{w'w} + \tau(w'w) \geq \ratw{w} = \pi(w)$.
	
	For the converse, let $t$ and $\pi$ satisfy the four specified properties.
	Then, inductively, $\pi(w) \leq \ratw{w}$ for all $w\in\V{\exgraph}$, which implies that $\ratw{w} \geq \pi(w) > \ldt[]{w}$ for all $w\in\boundaries_\tlpsymb$.
	For $w\in\starts_\tlpsymb$ this is clear and for the other nodes $w$ we note that \cref{prop:calW-rtl-2} implies that $t(w'w) \leq \rounddown{\pi(w')}{w'w}$ since $t(w'w)\in\calT_{w'w}$.
	Using this, we obtain $\pi(w) \leq t(w'w) + \tau(w'w) \leq \rounddown{\pi(w')}{w'w} + \tau(w'w) \leq \rounddown{\ratw{w'}}{w'w} + \tau(w'w)$.
	Since this holds for all predecessors $w'$ of $w$, $\pi(w) \leq \ratw{w}$.
	
	For the final part, let $w'w\in\A{\exgraph}$ with $w'=(u',t')$.
	Choose $\walksym\in\tlpsymb$ of minimum lateness $\ell$ among those whose time-expanded path $\spathsym_\walksym=w_0\ldots w_n$ contains $w'w$.
	Inductively, $t(w'w) \leq \pi(w') \leq \pi(w_0) + \tau(\spathsym_\walksym w') \leq \edt[]{w_0} + \tau(\spathsym_\walksym w') = t'$.
	
	For the other direction, $t(w'w) \geq \pi(w_n) - \tau(w'\spathsym_\walksym) > \ldt[]{w_n} - \tau(w'\spathsym_\walksym) = t' - \exc(w_n) = t'-\ell$.
	By \cref{used-buffer-lateness-minimal-tlp}, the suffix $w\spathsym_\walksym$ shows that $w$ is $\ell$-extendable.
	It cannot be $\ell'$-extendable for some $\ell'<\ell$ since, otherwise, we can extend the prefix $\spathsym_\walksym w$ by that suffix, preserving $w'w$, to obtain a minimal too-long path of lower lateness by \cref{used-buffer-lateness-minimal-tlp-equiv}.
	Thus, $\requirement(w)=\ell$ and $t(w'w)\in(t'-\requirement(w),t']$.
\end{proof}
Based on this reformulation, we can solve ORP with the mixed-integer program~\eqref{mip:graphbased}.
It contains continuous variables $\pi_w$ to decide $\pi(w)$ and binary variables $x_{w'wt}$ which represent the desired time points $t(w'w)$.
Let $I_a$ denote the union of all intervals $I_{w'w}$ where the underlying arc $a(w'w)$ is $a$.
\begin{subequations}\label{mip:graphbased}
	\begin{alignat}{20}
		& \min \; & \sum_{a \in \A{D}}\sum_{t \in I_a} y_{at} \\
		& \mathrm{s.t.} & \sum_{t \in I_{w'w}} x_{w'wt} & = &\;& 1 &\quad& \forall w'w \in \A{\exgraph}\\
		& & x_{w'wt} & \leq && y_{a(w'w)t} && \forall w'w \in \A{\exgraph},\ t \in I_{w'w}\\
		& & \sum_{t \in I_{w'w}} t x_{w'wt} & \leq && \ratw{w'} && \forall w'w \in \A{\exgraph}\\
		& & \sum_{t \in I_{w'w}} t x_{w'wt} + \tau({a(w'w)}) & \geq && \ratw{w} && \forall w'w \in \A{\exgraph}\\
		&& \ratw{w}& =&&  t & & \forall w = (u,t) \in \starts\\
		&& \ratw{w}& \geq && \ldt[]{u} + 1& & \forall w = (u,t) \in \boundaries\\
		&& y_{at}   &\in && \{0,1\} && \forall a \in \A{D},\ t \in I_a \\
		&& x_{w'wt} & \in &&  \{0,1\} && \forall w'w \in \A{\exgraph}\\
		&& \ratw{(u,t)} &\geq && 0 && \forall (u,t) \in \V{\exgraph}
	\end{alignat}
\end{subequations}

While this model can be significantly more compact than~\eqref{mip:pathbased}, its size is still strongly dependent on the scaling of the time parameters.
So it seems promising to use dynamic discretisation discovery to solve ORP as well, which is the topic of the next subsection.

\subsection{A Dynamic Discretisation Discovery Algorithm for the ORP}\label{subsec:solving-ORP-DDD}

Like before, we want to use DDD to replace the time points by intervals in order to obtain a smaller model.
Thus, we again use a time discretisation on the arcs of the original network, which we denote by $\calD$ here to distinguish it from our previous ones.
The relaxation will, once again, use the intervals obtained by the discretisation and allow behaviours that are realisable for some point in the interval.

We shall be fairly brief in this section since none of the ideas used here are really different from what we did for the service network design problem.
When differences appear, we point them out.

To get started, let us state the problem and the relaxation we are trying to solve by using the reformulation \cref{characterisation-IP-calW-rtl} (or the IP~\eqref{mip:graphbased}).
The objective of the ORP is to add as few time points as possible, so our goal is to \enquote{consolidate} as many arcs $w'w$ that correspond to the same arc $a(w'w)$ as possible and \enquote{solve} them by adding a single time point.
Thus, a solution is a partition $\orpsol$ of the arcs $\A{\exgraph}$ such that all arcs $w'w$ in a part $\phi\in\orpsol$ correspond to the same arc $a(w'w)$.
We write $\phi(w'w)$ for the part $\phi\in\orpsol$ containing $w'w$.

For feasibility, we need time points $t\colon \A{\exgraph}\to\NN$ and $\pi\colon \V{\exgraph}\to\NN$ that satisfy the conditions of \cref{characterisation-IP-calW-rtl}.
Since we want all arcs in the same part of $\orpsol$ to have the same $t$-value, it suffices to specify time points for each part of the partition.

Hence, $\orpsol$ is feasible if time points $t\colon \orpsol\to\NN$ and $\pi\colon \V{\exgraph}\to\NN$ exist such that
\begin{enumprop}
	\item $\pi(w)\leq \edt[]{w}$ for all $w\in\starts_\tlpsymb$,
	\item $t(\phi(w'w)) \leq \pi(w')$ for all $w'w\in\A{\exgraph}$,
	\item $t(\phi(w'w)) + \tau(w'w) \geq \pi(w)$ for all $w'w\in\A{\exgraph}$, and
	\item $\pi(w) > \ldt[]{w}$ for $w\in\boundaries_\tlpsymb$.
\end{enumprop}
Recall that these conditions imply that $t(\phi(w'w))\in I_{w'w}$.

In the relaxation \rorp, we use the time discretisation $\calD$ and assign each part~$\phi$ of the partition $\orpsol$ an interval $h(\phi)\in \calD_{a(\phi)}^I$, where $a(\phi)$ is the arc in $D$ corresponding to~$\phi$.
We now require that $h(\phi(w'w))\cap I_{w'w}\neq \emptyset$ for all $w'w\in\A{\exgraph}$ together with
\begin{enumprop}
	\item $\pi(w)\leq \edt[]{w}$ for all $w\in\starts_\tlpsymb$,
	\item $\min h(\phi(w'w)) \leq \pi(w')$ for all $w'w\in\A{\exgraph}$, \label{prop:orp-ddd-2}
	\item $\max h(\phi(w'w)) + \tau(w'w) \geq \pi(w)$ for all $w'w\in\A{\exgraph}$, and \label{prop:orp-ddd-3}
	\item $\pi(w) > \ldt[]{w}$ for $w\in\boundaries_\tlpsymb$.
\end{enumprop}

Clearly, this is a relaxation (simply assign the interval containing a time point) and it is tight for the full discretisation (then the interval minimum and maximum are just the same single point).
To solve it, we can use a mixed-integer program similar to~\eqref{mip:graphbased}, selecting intervals instead of time points and replacing the coefficient $t$ in the linking constraints by the upper and lower bounds of the intervals.

Now we need to find our \enquote{problems}, so an equivalent to the too-long paths in the dispatch-node graph, and how to refine the discretisation to eliminate them, so an equivalent of making them $\calT$-relaxed too-long.
We note that an analogous construction to the one we presented before would be plausible, though we deviate slightly from it.
Before we describe how we proceed in this setting, we want to describe the approach we take now in terms of the \snd.

In the \snd, we could have defined a different graph to verify implementability than the dispatch-node graph:
it would have also been possible to construct a \emph{consolidation graph} that has a node for every consolidation and connects consolidations that appear consecutively on a path.
Then analogous definitions of ($\calT$-relaxed) too-long paths can be derived as well as the corresponding characterisations.
The main difference between these two options is that the consolidation graph has a smaller size, but also \enquote{solves} fewer problems, resulting in potentially more iterations.
For the \snd, the dispatch-node graph was computationally advantageous, but for the ORP, the graph $\exgraph$ is significantly larger, so we opt for the more compact representation given by the consolidation graph.

Since the consolidations correspond to parts of the partition here, the graph defined below has these as its nodes.
\begin{definition}[Partition-linking graph]
	Given a solution $\orpsol$, its \emph{partition-linking graph} $D_\orpsol$ is  given by
	\begin{align*}
		\V{D_\orpsol} &= \orpsol \\
		\A{D_\orpsol} &= \Set{\phi(w_1w_2)\phi(w_2w_3) \mid w_1w_2,w_2w_3 \in \A{\exgraph} }.
	\end{align*}
	For arcs $\phi'\phi$ we set $\tau(\phi'\phi) = \tau(a(\phi'))$ and for nodes $\phi$ we set
	\[
		\edt[]{\phi} = \max \Set{\min I_{w'w} : w'w \in \phi},
		\qquad
		\ldt[]{\phi} = \min\Set{\max I_{w'w} : w'w \in \phi}.
	\]
\end{definition}
Note that in a feasible solution of \orp with time points $t$ and $\pi$, $t(\phi)\in I_{w'w}$ for all $w'w\in \phi$, so $t(\phi)\in [\edt[]{\phi},\ldt[]{\phi}]$.
Similarly, in a feasible solution of \rorp with time points $h$ and $\pi$, $h(\phi)\cap I_{w'w} \neq \emptyset$ for all $w'w\in \phi$, so $\max h(\phi)\geq \edt[]{\phi}$ and $\min h(\phi) \leq \ldt[]{\phi}$.

Since we now want to be sufficiently late at the boundary nodes, the obstacle to being a feasible solution is that paths in the partition-linking graph might be too short.
\begin{definition}
	A $\phi'\phi$-path $\spathsym$ in $D_\orpsol$ is \emph{too-short} if $\ldt[]{\phi'} + \tau(\spathsym) < \edt[]{\phi}$.
\end{definition}

These too-short paths are the obstruction to implementability.
\begin{theorem}
	A solution $\orpsol$ is feasible for the ORP if and only if $D_\orpsol$ contains no too-short path.
\end{theorem}
\begin{proof}
	Let $\orpsol$ be feasible with time points $t$ and $\pi$.
	Moreover, let $\spathsym=\phi_1\ldots\phi_n$ be a path in $D_\orpsol$.
	For any arc $\phi_i\phi_{i+1}$ we get nodes $w_1,\, w_2,\, w_3\in\V{\exgraph}$ such that $\phi_i=\phi(w_1w_2)$ and $\phi_{i+1}=\phi(w_2w_3)$.
	By feasibility, $t(\phi_i) = t(\phi(w_1w_2)) \geq \pi(w_2) - \tau(w_1w_2) \geq t(\phi(w_2w_3)) - \tau(a(\phi_i)) = t(\phi_{i+1}) - \tau(\phi_i\phi_{i+1})$.
	Thus, $\ldt[]{\phi_1} + \tau(\spathsym) \geq t(\phi_1) + \tau(\spathsym) \geq t(\phi_n) \geq \edt[]{\phi_n}$, so $\spathsym$ is not too-short.
	
	Conversely, let $D_\orpsol$ contain no too-short path.
	In this case, we can define $t(\phi) = \min\Set{\ldt[]{\phi'} + \dist[\tau]{\phi',\phi} : \phi'\in\orpsol} \geq \edt[]{\phi}$ for $\phi\in\orpsol$.
	We also set $\pi(w) = \edt[]{w}$ for $w\in\starts_\tlpsymb$, and $\pi(w) = \min\Set{t(\phi(w'w)) + \tau(w'w) : w'\in\Nin{w}}$.
	The triangle inequality yields the following helpful inequality:
	for $\phi'\phi\in\A{D_\orpsol}$
	\begin{equation}
		\label{eq:triangle-inequality-orp}
		t(\phi) \leq t(\phi') + \tau(\phi'\phi).
	\end{equation}
	
	Using these values, $\pi(w)=\edt[]{w}$ for $w\in\starts_\tlpsymb$.
	Next, let $w'w\in\A{D_\tlpsymb}$ and $\phi=\phi(w'w)$.
If $w'\in\starts_\tlpsymb$, we get $t(\phi) \leq \ldt[]{\phi} \leq \max I_{w'w} = \edt[]{w'} = \pi(w')$.
	Otherwise, $\pi(w') = t(\phi(w''w')) + \tau(w''w')$ for some $w''\in\Nin{w'}$.
	Let $\phi(w''w') = \phi'$. Then $\phi'\phi\in\A{D_\orpsol}$, giving us $t(\phi) \leq t(\phi') + \tau(\phi'\phi) = \pi(w')$, by \eqref{eq:triangle-inequality-orp}.
	Also, $\pi(w) \leq t(\phi) + \tau(w'w)$ by definition.
	
Finally, for $w=(u,t)\in\boundaries_\tlpsymb$, we get a $w'=(u',t')\in\Nin{w}$ with $\pi(w) = t(\phi(w'w)) + \tau(w'w) \geq \edt[]{\phi(w'w)} + \tau(w'w) \geq \min I_{w'w} + \tau(w'w) > t' + \tau(w'w) - \requirement(w) = t - \exc(w) = \ldt[]{w}$.
\end{proof}

Just as for refining the discretisation $\calT$ for \rsnd, we have many options available for refining the discretisation $\calD$ for \rorp.
In fact, we could transfer the concepts we developed to solve \orp to finding small discretisations $\calD$ for \rorp.
But in preliminary experiments we found that empirically, a very simple initial discretisation was always sufficient that the \rorp solution was also feasible for \orp.
So we simply describe this initialisation and a sufficient refinement approach (which was never necessary in our experiments).

We initialise $\calD_a = \bigcup_{w'w\in \A{D_\tlpsymb} \colon a(w'w) = a} \Set{\min I_{w'w}, \max I_{w'w} +1}$, i.e., by splitting at the interval endpoints for all $ww'$ that need to select a time point on arc $a$.
If the solution of \rorp is nonetheless infeasible for \orp, we could find (e.g., by enumeration) one or more too-short paths and then add time points according to the following statement, which is conceptually equivalent to~\cref{theorem:real_arrival_time_eliminates}:
\begin{lemma}\label{lemma:orp_ddd_time_points}
	Let $\orpsol$ be a solution to \orp, $\spathsym = \phi_1\ldots\phi_n$ be a too-short path in $D_\orpsol$, and $t_i = \ldt[]{\phi_1} + \tau(\spathsym\phi_i)$.
	If $\calD$ is a discretisation with $t_i + 1 \in \calD_{a(\phi_i)}$ for $i \in (n]$, then $\orpsol$ is not feasible for \rorp.
\end{lemma}
\begin{proof}
	Suppose $\orpsol$ was feasible for \rorp, then we obtain intervals $h(\varphi)\in\calD_{a(\phi)}^I$ that satisfy all five conditions in the definition of the relaxation.
	For any arc $\phi_i\phi_{i+1}$ we get nodes $w_1,\, w_2,\, w_3\in\V{\exgraph}$ such that $\phi_i=\phi(w_1w_2)$ and $\phi_{i+1}=\phi(w_2w_3)$.
	By feasibility, $\max h(\phi_i) = \max h(\phi(w_1w_2)) \geq \pi(w_2) - \tau(w_1w_2) \geq \min h(\phi(w_2w_3)) - \tau(a(\phi_i)) = \min h(\phi_{i+1}) - \tau(\phi_i\phi_{i+1})$.
	
	Note that $\min h(\phi_1) \leq \ldt[]{\phi_1} = t_1$ and $t_1+1\in\calD_{a(\phi_1)}$, so $\max h(\phi_1) \leq t_1$ as well.
	Now, inductively, $\min h(\phi_{i+1}) \leq \max h(\phi_i) + \tau(\phi_i\phi_{i+1}) \leq \ldt[]{\phi_1} + \tau(\spathsym \phi_{i+1}) = t_{i+1}$.
	Again, since $t_{i+1}+1$ is in the discretisation of $a(\phi_{i+1})$, we get $\max h(\phi_{i+1}) \leq t_{i+1}$.
	In particular, $\max h(\phi_n) \leq t_n < \edt[]{\phi_n} \leq \max h(\phi_n)$, which is a contradiction.
\end{proof}

%% file: sections/variants.tex
\section{Discussion of Some Modifications}\label{sec:variants}

In this section, we discuss several modifications to the proposed framework and how our results carry over.
The first modification is a slightly different definition of (relaxed) too-long paths, the second a stronger relaxed problem, and the third an initialisation strategy for the discretisation.

\subsection{Alternative (\calT-Relaxed) Too-Long Criteria}
The following definition is similar to the elimination criterion in \cite[][Lemma 2]{shu_new_2026}, though slightly different since we again do not restrict too-long paths to start at commodity origins.

\begin{definition}[(\calT-Relaxed) too-long$^*$ paths]
	Let $uu'$ be an arc in $D_\solsym$ with $u=\vdn{v}{k}$, and $u'=\vdn{v'}{k'}$.
	We define
	\[
		\edt[]{uu'} = \max\Set{\edt[k]{v},\edt[k']{v}},
		\qquad
		\ldt[]{uu'}=\min\Set{\ldt[k]{v'},\ldt[k']{v'}}.
	\]
	
	We also say that a walk $\walksym=u_0\ldots u_n$ in $D_\solsym$ is a \emph{too-long$^*$ path} if $\edt[]{u_0u_1} + \tau(\walksym) > \ldt[]{u_{n-1}u_n}$.
	
	Similarly, we change the initialisation of the relaxed arrival times and set $\ratp{0} = \edt[]{u_0u_1}$ and say $\walksym$ is $\calT$-relaxed too-long$^*$ if $\ratp{n} > \ldt[]{u_{n-1}u_n}$.
\end{definition}
Note that we continue to assume that single-node too-long paths do not exist, so this new definition makes sense.
Essentially, this definition makes use of the observation that if two commodities are consolidated on an arc, they must also both traverse that arc, and we can disregard the less restrictive commodity.

\begin{example}
	Consider the partial solution in~\cref{example:sharedarc:routes-2}.
	Assume other paths exist between the nodes such that the unspecified values of $\edt{v},\ldt{v}$ are irrelevant.
    There exist two (minimal) too-long paths: $(v_1,k_1)(v_2,k_1)(v_3,k_1)$ and $(v_1,k_1)(v_2,k_1)(v_3,k_2)$. 
    Assume the first arc is fully discretised in $\calT$; then under our definition, we require one time point in $[2,6]$ on the second arc to make the first path $\calT$-relaxed too-long and one time point in $[4,6]$ to make the second path $\calT$-relaxed too-long.
    In contrast, we require only a time point in $[2,6]$ to make both paths $\calT$-relaxed too-long$^*$.
    Evidently, this is less restrictive and, depending on the other too-long paths, a smaller discretisation suffices under this definition.
	
	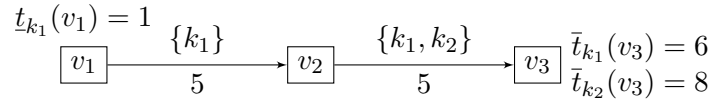
\begin{figure}[htbp]
		\centering
		\begin{tikzpicture}[x=3cm,y=1.2cm]
			\node[draw,label={[align=left]above:{$\edt[k_1]{v_1}=1$}}] (v1) at (0,0) {$v_1$};
			\node[draw] (v2) at (1,0) {$v_2$};
			\node[draw,label={[align=left]right:{$\ldt[k_1]{v_3}=6$\\$\ldt[k_2]{v_3} = 8 $}}] (v3) at (2,0) {$v_3$};
            \draw[->] (v1) to node[above] {$\Set{k_1}$} node[below] {$5$}(v2);
			\draw[->] (v2) to node[above] {$\Set{k_1,k_2}$} node[below] {$5$} (v3);
		\end{tikzpicture}
		\caption{Example (partial) solution to demonstrate that fewer time points might suffice.}
		\label{example:sharedarc:routes-2}
	\end{figure}
\end{example}

Any ($\calT$-relaxed) too-long path is also a ($\calT$-relaxed) too-long$^*$ path.
The statements regarding implementable and representable solutions (\cref{characterisation:implementable,characterisation:representable}) carry over, so we can also transfer the path-based refinement strategies.
It is also possible to construct a graph similar to $\exgraph$ and show a statement corresponding to~\cref{criterion:tetlpg}.
But because the modified elimination criterion considers the minimum over the last two commodities, the graph construction is more complicated.
Preliminary experiments indicated that using this stronger criterion has little impact on the resulting discretisations and overall running times, so for clarity of exposition we use the simpler criterion.

\subsection{Stronger Relaxation}
Next, we discuss a possible modification to the relaxed problem R-SND.
For a given arc $a$ and interval $I$ we define the part of the interval in which a commodity $k$ can actually be dispatched by $J_k(a,I) = I \cap I_k(a)$.
Then, we can replace \cref{def:representable:linking} for a representable solution by the more restrictive condition
\begin{equation*}
   \min J_k(a_i,h(c_i)) + \tau(a_i) \leq \max J_k(a_{i+1},h(c_{i+1})) \quad \forall i \in (n).
\end{equation*}
This essentially ensures that commodity $k$ cannot be dispatched on $a$ in $I$ and immediately after on $a'$ in $I'$ if such a transfer is not possible while respecting the commodity's dispatch windows, even if a dispatch in either interval by itself is possible.
We also describe how this modified problem can be modelled as an integer program in~\cref{appendix:IP}.
The commodity-specific subnetworks used for the relaxations in recent works~\cite{helber_arc-based_2026,marshall_interval-based_2021,shu_new_2026} reflect this rule for allowed arc sequences.
Evidently, this modified relaxation is more restrictive and every solution that is not representable for R-SND is also not representable under the modification.
So all of our refinement strategies are directly applicable when using this relaxation as well.

Similarly, when discussing DDD for the optimal refinement problem in \cref{subsec:solving-ORP-DDD}, we could have also strengthened the relaxation by not only requiring $J_{w'w} = h(\phi(w'w))\cap I_{w'w}\neq \emptyset$, but also using this more restrictive interval in \cref{prop:orp-ddd-2,prop:orp-ddd-3}.
These would then become
\begin{enumprop}[start=2]
	\item $\min J_{w'w} \leq \pi(w')$ for all $w'w\in\A{\exgraph}$ and
	\item $\max J_{w'w} + \tau(w'w) \geq \pi(w)$ for all $w'w\in\A{\exgraph}$.
\end{enumprop}
In our computational experiments, we use these stronger relaxations when solving both SND and ORP with DDD.

\subsection{Initial Discretisation}\label{variant:significant_timepoints}
Lastly, we discuss how the significant time points of~\textcite{shu_new_2026} can be integrated into our refinement approaches.
For an arc $a \in \A{D}$ we define the set of commodities that can traverse this arc in a feasible solution by $\calK_{a} = \Set{k \in \calK : \edt{a} \leq \ldt{a}}$ and then observe that some infeasible consolidations are easy to avoid:
\begin{theorem}[Significant time points, {{\cite[][Theorem 1]{shu_new_2026}}}]
\leavevmode\\
Given an arc $uv \in \A{D}$ and two different commodities $k_1,k_2 \in \calK_{uv}$ with $\edt[k_2]{u} + \tau(uv) > \ldt[k_1]{v}$, if there exists a time point $t \in \calT_{uv}$ with $t \in (\ldt[k_1]{v}- \tau(uv), \edt[k_2]{u}]$, then $k_1$ and $k_2$ cannot be consolidated in any $\calT$-representable solution.
\end{theorem}
By adding such time points to the discretisation, we essentially directly eliminate all single-arc minimal too-long paths.
\textcite{shu_new_2026} propose solving a minimum hitting set problem to identify the smallest set of time points that eliminate all these paths and use this set to initialise the discretisation.

Specifically, for each arc $uv$ we want a minimum-size set of time points that hits all intervals in $\mathcal{I}_{uv} = \Set{(\ldt[k_1]{v}-\tau(uv),\edt[k_2]{u}] : k_1,k_2 \in \calK_{uv}, \edt[k_2]{u} + \tau(uv) > \ldt[k_1]{v}}$.
We could require these time points to remain in the discretisation during refinement, or instead add all corresponding one-arc paths to the set $\tlpsymb$ for the path-based approaches or into the graph $\exgraph$.
Since these paths can be quite numerous, we instead propose the following approach (for each arc).
First, we solve the hitting set problem, obtaining a set of time points $T_{uv}$.
We then assign each interval $I$ to one of the selected points it contains, which we call $t(I)$.
Evidently, we could replace a $t \in T_{uv}$ with any $t' \in I(t) = \bigcap_{I \in \mathcal{I}_{uv}\colon t(I) = t} I$ and still hit every interval.
So in the initial discretisation, we simply set $\calT_{uv} = \{0,H\} \cup T_{uv}$ and during refinement, we add a constraint to our integer programs that the discretisation needs to include at least one $t' \in I(t)$ for each $t \in T_{uv}$. 
Essentially, we allow the significant time points to be repositioned to help eliminate newly found conflicts.

%% file: sections/computational-study.tex
\section{Computational Study}
We present the results of two groups of experiments.
First, we demonstrate the efficiency of our exact methods for solving instances of the optimal refinement problem and of our algorithm for constructing~\(\exgraph\) in Section~\ref{subsec:orpsolvers}.
Then, we investigate the benefit of optimal refinement for the overarching DDD approach when solving \snd in Section~\ref{subsec:sndsolvers}.

We implemented all algorithms in Python 3.11.4 and solved all mixed-integer programs with Gurobi 12.0.0 \cite{gurobi}.
The implementation and data will be made publicly available upon acceptance of this manuscript.
All experiments were conducted on machines with a Xeon L5630 Quad Core \SI{2.13}{\giga\hertz} processor.

We utilise the instance sets for the \snd proposed by \textcite{boland_continuous-time_2017} (and classified in~\cite{marshall_interval-based_2021}), \textcite{van_dyk_sparse_2024}, and \textcite{helber_arc-based_2026}.
\cref{tab:instance-characteristics} gives an overview of the size of these instances. For more details we refer to the corresponding papers.
Note that we omit the LC/LF and LC/HF instances of \textcite{boland_continuous-time_2017}, since almost all are solved in a single iteration with state-of-the-art DDD solvers, so no refinement takes place.

\begin{table}[tbp]
    \centering
    \caption{Characteristics of the \snd instance sets used.}
    \label{tab:instance-characteristics}
    \begin{tabular}{llrrrr}
        \toprule
        Source & Class & Instances & Commodities & Nodes & Arcs \\
        \midrule
        \cite{boland_continuous-time_2017} &HC/HF           & 177 & 100--400 & 20--30 & 228--683\\
        &HC/LF           & 183 & 100--400 & 20--30 & 229--686 \\
        \midrule
        \cite{van_dyk_sparse_2024}&Critical times  & 192 & 150--200 & 20     & 230--300 \\
        &Designated paths& 192 & 150--300 & 20--25 & 230--480 \\
        &Hub-and-spoke (easy)   & 192 & 100      & 20     & 45--95   \\
        \midrule
        \cite{helber_arc-based_2026}&Grid            & 100 & 50       & 16     & 48        \\
        &Hub-and-spoke (hard) & 100 & 75       & 30     & 120--204 \\
        &Star            & 100 & 100      & 21     & 40       \\
        &Wheel           & 100 & 100      & 12     & 44       \\
        \bottomrule
    \end{tabular}
\end{table}

\subsection{Comparing Solvers for ORP}\label{subsec:orpsolvers}
To generate instances of ORP, we ran the DDD solver for one iteration on each instance with a time limit of one hour and a relative optimality-gap limit of \SI{1}{\percent} and stored the dispatch-node graph $D_\solsym$ of the resulting relaxation solution $\solsym$.
The initial discretisation for this first iteration contained the significant time points as described by \textcite{shu_new_2026}.
We made this choice to obtain ORP instances that are representative of those encountered in state-of-the-art DDD approaches.

\cref{tab:orp-instance-characteristics} reports the characteristics of the resulting ORP instances. 
About \SI{90}{\percent} of graphs contained fewer than \num{10000} minimal too-long paths, about \SI{4}{\percent} contained between \num{10000} and \num{50000}, about \SI{1}{\percent} contained between \num{50000} and \num{100000}, but about \SI{5}{\percent} contained more than \num{100000}.
For graphs with more than \num{100000} minimal too-long paths, enumerating them almost always exceeded the memory limit (\SI{16}{\giga\byte}), likely due to exponential blow-up.
We therefore capped enumeration at \num{100000} minimal too-long paths for approaches where this is relevant; the last column reports the number of instances that hit this cap in each class.

\begin{table}[tbp]
    \centering
    \caption{Characteristics of derived ORP instance sets.}
    \label{tab:orp-instance-characteristics}
    \begin{tabular}{llrrrr}
        \toprule
        Source & Class & Instances & Nodes & Arcs & $|\tlpsymb| \geq 10^5$ \\
        \midrule
        \cite{boland_continuous-time_2017}   & HC/HF            & 177 & 298--1261 & 317--4041  & 0  \\
                 & HC/LF            & 149 & 273--1242 & 260--2982  & 0  \\
        \midrule
        \cite{van_dyk_sparse_2024}      & Critical times   & 188 & 455--732  & 1143--4001 & 0  \\
                 & Designated paths & 192 & 490--1296 & 852--8196  & 51 \\
                 & Hub-and-spoke (easy)  & 158 & 293--476  & 1091--2756 & 7  \\
        \midrule
        \cite{helber_arc-based_2026} & Grid             & 87  & 151--204  & 259--524   & 0  \\
                 & Hub-and-spoke (hard)  & 84  & 250--311  & 440--729   & 0  \\
                 & Star             & 100 & 284--299  & 482--653   & 0  \\
                 & Wheel            & 100 & 281--316  & 520--712   & 0  \\
        \bottomrule
    \end{tabular}
\end{table}

Note that the instance counts are lower for some classes than the number of available \snd instances.
We excluded 30 of the (easy) hub-and-spoke instances because they had arcs with $\tau(a) = 0$, which may lead to infinitely many minimal too-long paths.
For the omitted grid and hard hub-and-spoke instances, the solver ran into the one-hour time limit while solving the first relaxation.
The remaining instances were omitted because the dispatch-node graph obtained after the first iteration contained no too-long paths and was thus not relevant for the ORP.

We first compare the efficiency of \cref{algo:exgraph} for deriving $\exgraph$ from $D_\solsym$ to the naive approach of first enumerating all minimal too-long paths and then merging them.
The box plot in \cref{fig:exgraph_creation_boxplot} summarises the distribution of running times we observed for instances in each class.
We observe that for most classes, the mean and the quartiles are similar between the two construction methods.
Recall that most of our dispatch-node graphs contain (significantly) fewer than \num{10000} minimal too-long paths, so it is not surprising that enumerating them is not a large burden.
But for some classes, we observe that path-based construction may lead to outliers with significantly longer running times. 
Importantly, due to our enumeration limit of \num{100000} paths, the path-based approach did not create the full \(\exgraph\) on the 58 instances in which it ran into the limit, while our graph-based algorithm was able to create the full graph every time.
So our time-expanded too-long path graph allows us to represent all conflicts even when the path enumeration balloons past the memory limit.
\begin{figure}[htbp]
    \centering
	\input{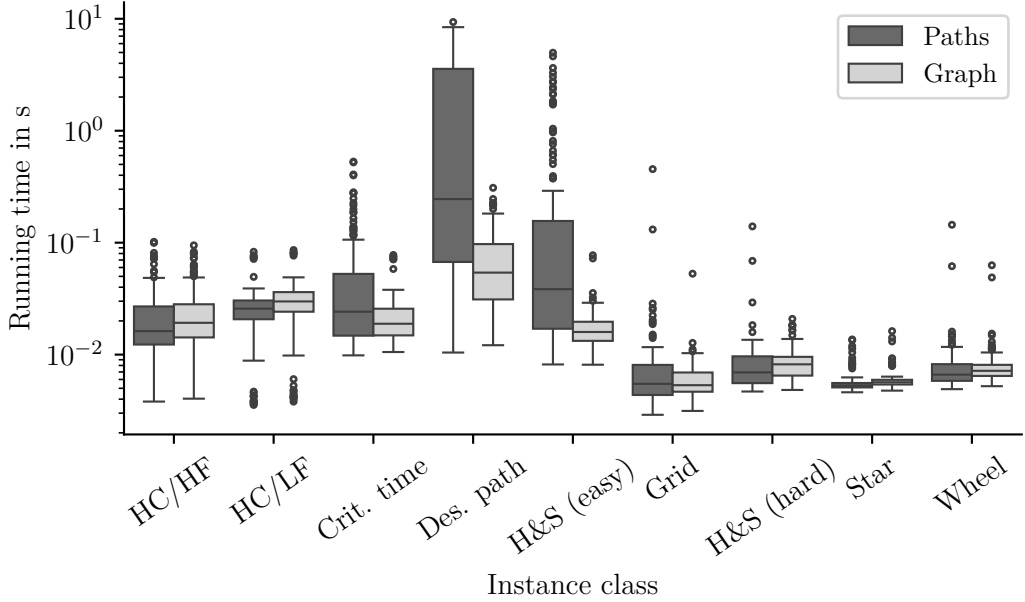} 
    \caption{Time to create $\exgraph$ by enumerating and merging all minimal too-long paths (Paths) or by our~\cref{algo:exgraph} (Graph).}
    \label{fig:exgraph_creation_boxplot}
\end{figure}

Next, we compare the efficacy of various approaches for solving the ORP.
As a baseline (BL) heuristic, we add for each enumerated minimal too-long path time points according to its real arrival time, which by~\cref{theorem:real_arrival_time_eliminates} will eliminate it.
We also use the matheuristic (MH) of \textcite{shu_new_2026} adapted to the arc-based discretisation. 
Recall that this heuristic uses their stronger elimination criterion and so in principle could find smaller discretisations than our approaches. 
For exact optimisation, we use the path-based MIP~\cref{mip:pathbased} (Path), the graph-based MIP~\cref{mip:graphbased} (Graph), and our graph-based DDD approach (DDD).
Note that Path is only exact for instances where all minimal too-long paths could be enumerated.

\cref{tab:orp-exp2} reports the mean running time and solution quality per class.
The running time also includes the time for enumerating too-long paths or constructing $\exgraph$.
Recall that for some instances, not all minimal too-long paths were enumerated; in principle, the path-based approaches (BL, MH, Path) may produce solutions for these instances that do not eliminate all paths (and thus are infeasible).
This affects 78 designated path, 29 (easy) hub-and-spoke, and 9 critical time instances.
The numbers here are higher than in the previous experiments because MH uses the too-long path definition of~\textcite{shu_new_2026}, which yields additional too-long paths (that would be non-minimal under our definition).
We retain these instances in~\cref{tab:orp-exp2-runtime}, since we still obtain useful information about the time it takes to compute a discretisation that eliminates many too-long paths. 
But for~\cref{tab:orp-exp2-points}, we exclude them, since the number of time points cannot sensibly be compared with the number for feasible solutions.

We observe that BL adds many more time points than are strictly necessary, with discretisation sizes between \SI{20}{\percent} and \SI{68}{\percent} above the optimum, while the matheuristic significantly reduces this gap, even reaching near-optimal solutions for the star instances.
But the graph-based MIP and DDD approach achieve the optimal solution values significantly faster on average than MH and Path.
For some classes, DDD is even faster than the simple BL heuristic, since it never requires enumeration of all minimal too-long paths.
We also depict the performance profile for the three exact algorithms in~\cref{fig:orp-exp2-perfprof}.
DDD is almost always the fastest, with Graph and Path taking at least 10 times as long as DDD on about \SI{20}{\percent} and \SI{50}{\percent} of instances, respectively.
Due to the initial time points, DDD always finished in one iteration.
In summary, for the ORP instances resulting from \snd instances from the literature, our DDD approach allows us to find optimal refinements with mean running times below \SI{1.3}{\second} in every instance class.

\begin{table}[htbp]
	\footnotesize
	\centering
	\caption{Comparison of ORP solution methods.}
	\label{tab:orp-exp2}
	\begin{subtable}{0.55\linewidth}
		\centering
		\caption{Mean running time in \unit{\second}.}
		\label{tab:orp-exp2-runtime}
		\begin{tabular}{lrrrrr}
			\toprule
& BL & MH & Path & Graph & DDD \\
			Class &  &  &  &  &  \\
			\midrule
			HC/HF & 0.01 & 10.06 & 27.51 & 7.34 & 0.19 \\
			HC/LF & 0.01 & 3.38 & 2.98 & 1.75 & 0.16 \\
			Crit. time & 0.05 & 4.43 & 2.02 & 0.11 & 0.08 \\
			Des. path & 1.36 & 44.72 & 56.96 & 1.36 & 1.27 \\
			H\&S (easy) & 0.39 & 6.96 & 16.57 & 0.17 & 0.12 \\
			Grid & 0.01 & 6.34 & 1.12 & 0.10 & 0.05 \\
			H\&S (hard) & 0.01 & 3.40 & 0.66 & 0.19 & 0.05 \\
			Star & 0.00 & 0.06 & 0.09 & 0.08 & 0.01 \\
			Wheel & 0.00 & 1.74 & 0.36 & 0.18 & 0.03 \\
			\bottomrule
		\end{tabular}
	\end{subtable}
	\begin{subtable}{0.4\linewidth}
		\centering
		\caption{Mean percentage excess in discretisation size relative to the ORP optimum.}
		\label{tab:orp-exp2-points}
		\begin{tabular}{lrrr}
			\toprule
			& BL & MH  \\
			Class &  &    \\
			\midrule
			HC/HF & 67.98 & 18.62\\
			HC/LF & 56.94 & 14.48\\
			Crit. time & 25.05 & 10.76\\
			Des. path & 60.24 & 28.79\\
			H\&S (easy) & 20.43 & 8.57 \\
			Grid & 40.64 & 11.54 \\
			H\&S (hard) & 51.38 & 6.41 \\
			Star & 36.66 & 0.04 \\
			Wheel & 36.58 & 2.41 \\
			\bottomrule
		\end{tabular}
	\end{subtable}
\end{table}

\begin{figure}[htbp]
    \centering
    
    \input{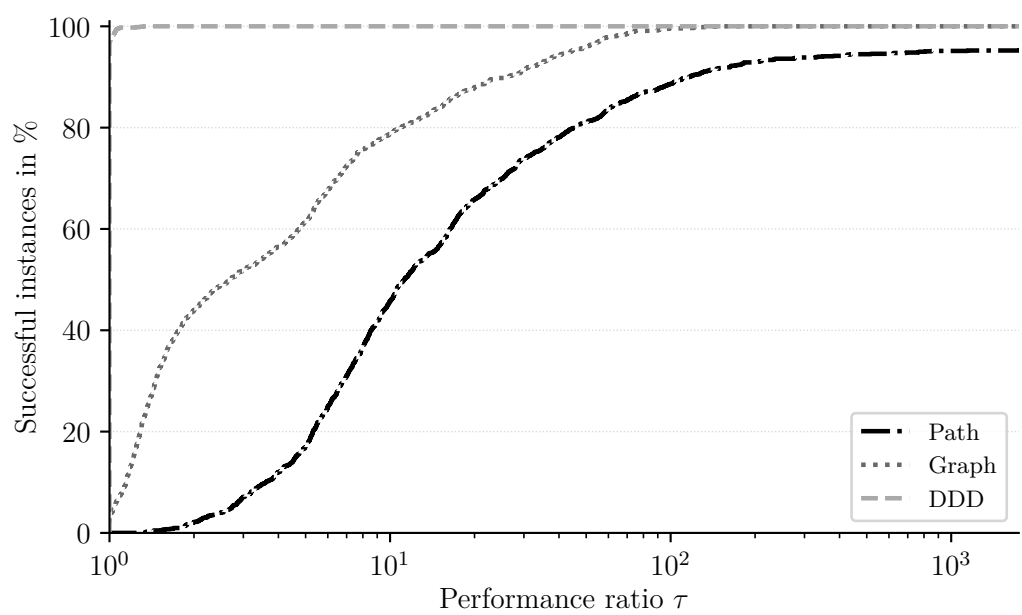}
    \caption{Performance profile for exact ORP solvers: percentage of instances solved to optimality within $\tau$ times the running time of the fastest solver on each instance.}
    \label{fig:orp-exp2-perfprof}
\end{figure}

\subsection{Benefit of Optimal Refinement for Solving \snd}\label{subsec:sndsolvers}

Next, we investigate the impact of different refinement strategies on the performance of the DDD algorithm for the \snd.
For these experiments, we disregard the instances of~\textcite{van_dyk_sparse_2024}, since they are trivial to solve with existing methods~\cite{helber_arc-based_2026}.
We compare various refinement strategies in a shared DDD algorithm.
This algorithm is based on the stronger relaxed problem discussed in Section~\ref{sec:variants} and uses various features of state-of-the-art algorithms.
We use the repair heuristic of~\textcite{marshall_interval-based_2021} to obtain feasible solutions to \snd in each iteration and use their adaptive gap to only solve the relaxations to certain solution quality. 
To ensure that the heuristic always finds a feasible solution, we also add a constraint to enforce $\tau(\spathsym_k) \leq \ell_k - r_k$ for each commodity $k$ \cite{boland_continuous-time_2017}.
The algorithm also uses every solution found by the relaxation solver for refinement and repair, as suggested by~\textcite{shu_new_2026}, and initialises the relaxation with their significant time points.
Each instance is solved with a gap target of \SI{1}{\percent} and time limit of one hour.

Again, we use the baseline (BL) strategy of simply adding the time points of~\cref{theorem:real_arrival_time_eliminates} for all minimal too-long paths, as well as the matheuristic (MH) of \textcite{shu_new_2026}.
Both approaches simply add time points to the existing discretisation.
Similarly, we introduce an approach that uses the DDD solver for the ORP to find the smallest set of time points to add to the existing discretisation to eliminate all newly observed too-long paths (Iterative). 
As mentioned in the literature overview, we also investigate strategies that can remove time points from the discretisation.
We propose two methods using this idea.
The method Parallel finds a smallest discretisation $\calT$ that makes all too-long paths in the dispatch-node graphs of all relaxation solutions obtained so far \calT-relaxed too-long.
This makes the refinement problem substantially larger and potentially harder to solve, even with the DDD approach.
This motivates the method Merged, which takes a similar approach but merges the dispatch-node graphs, reducing the refinement problem's size at the possible cost of additional time points.
For Parallel and Merged, we also require them to include significant time points as discussed in Section~\ref{variant:significant_timepoints}.

\cref{tab:orp-exp3} reports statistics for each approach.
Note that we did not include the instances that were solved in one iteration or for which the first relaxation could not be solved within the time limit, since refinement does not impact their performance (i.e., instance counts per class are as in \cref{tab:orp-instance-characteristics}).
Note, also, that the model size refers to the number of commodity transport variables relative to the number in a fully time-expanded model, as in \cite{helber_arc-based_2026}.
The average number of iterations for instances in each class only varied by about \num{0.2} between the approaches and is not reported in detail.
The differences in running time primarily reflect differences in the difficulty of solving the relaxations rather than differences in iteration counts due to weaker or stronger relaxations.
We see that approaches using optimal refinement (Parallel, Merged, Iterative) achieve the lowest running times for each instance class and solve slightly more instances than the heuristic approaches.
But overall, the gains in mean running time are modest. For example, Iterative uses \SI{12}{\percent} and \SI{23}{\percent} less time on average than MH and BL, respectively, on the HC/HF instances.
We observe that Parallel usually achieves the smallest model sizes at termination, almost \SI{20}{\percent} smaller than the naive BL approach on the HC/HF class.
Merged achieves essentially the same size while requiring less time for refinement.
But again, the decreases in model size are modest, with the largest decrease of \SI{18}{\percent} from \SI{0.76}{\percent} of a fully time-expanded model to \SI{0.62}{\percent} for Parallel versus BL on HC/HF.

We note that we also conducted experiments in which we solved the ORP with a weighted objective function to prefer arc time points that lead to timed copies for fewer commodities.
This resulted in negligible differences in model size and running time; therefore we omit the detailed results here.

\begin{table}
	\footnotesize
	\centering
	\caption{Mean results per class for each refinement strategy.}
	\label{tab:orp-exp3}
	\begin{tabular}{llrrrr}
		\toprule
		&  & Solved (\%) & Model size (\%) & Ref. (s) & Time (s) \\
		Class & Method &  &  &  &  \\
		\midrule
		\multirow[t]{5}{*}{HC/HF} & BL & 97.7 & 0.76 & 0.53 & 252.83 \\
		& MH & 99.4 & 0.67 & 19.78 & 221.47 \\
		& Parallel & 98.9 & 0.62 & 12.19 & 203.85 \\
		& Merged & 100.0 & 0.62 & 5.87 & 205.15 \\
		& Iterative & 100.0 & 0.65 & 2.33 & 194.12 \\
		\cline{1-6}
		\multirow[t]{5}{*}{HC/LF} & BL & 100.0 & 2.06 & 0.28 & 30.74 \\
		& MH & 100.0 & 1.89 & 8.13 & 35.96 \\
		& Parallel & 100.0 & 1.80 & 5.02 & 32.01 \\
		& Merged & 100.0 & 1.80 & 3.51 & 26.85 \\
		& Iterative & 100.0 & 1.84 & 1.37 & 26.71 \\
		\cline{1-6}
		\multirow[t]{5}{*}{Grid} & BL & 73.6 & 20.62 & 0.19 & 1425.75 \\
		& MH & 74.7 & 19.34 & 24.98 & 1412.90 \\
		& Parallel & 82.8 & 18.40 & 2.59 & 1280.40 \\
		& Merged & 82.8 & 18.41 & 1.32 & 1244.41 \\
		& Iterative & 78.2 & 18.90 & 0.65 & 1306.17 \\
		\cline{1-6}
		\multirow[t]{5}{*}{H\&S (hard)} & BL & 70.2 & 9.49 & 0.21 & 1565.40 \\
		& MH & 75.0 & 9.22 & 19.70 & 1444.95 \\
		& Parallel & 77.4 & 9.09 & 6.36 & 1431.72 \\
		& Merged & 76.2 & 9.10 & 4.15 & 1456.04 \\
		& Iterative & 77.4 & 9.16 & 1.10 & 1446.40 \\
		\cline{1-6}
		\multirow[t]{5}{*}{Star} & BL & 100.0 & 15.28 & 0.04 & 20.89 \\
		& MH & 100.0 & 14.77 & 7.22 & 24.20 \\
		& Parallel & 100.0 & 14.57 & 1.01 & 18.93 \\
		& Merged & 100.0 & 14.58 & 0.96 & 18.79 \\
		& Iterative & 100.0 & 14.60 & 0.21 & 16.42 \\
		\cline{1-6}
		\multirow[t]{5}{*}{Wheel} & BL & 85.0 & 5.14 & 0.16 & 1473.16 \\
		& MH & 89.0 & 4.76 & 5.43 & 1269.75 \\
		& Parallel & 93.0 & 4.59 & 2.26 & 1112.51 \\
		& Merged & 92.0 & 4.58 & 1.17 & 1076.32 \\
		& Iterative & 91.0 & 4.71 & 0.66 & 1199.93 \\
		\cline{1-6}
		\bottomrule
	\end{tabular}
\end{table}

%% file: sections/conclusion.tex
\section{Conclusion}
We introduced a more compact representation of conflicts in a DDD approach, which we used to design efficient methods for finding small discretisations that eliminate all conflicts.
Solving our refinement problem optimally leads to smaller relaxation problems that can, on average, be solved faster without a significant change in iteration count.
But, as reported by~\cite{helber_arc-based_2026}, the growth in relaxation size from the first to the final iteration for state-of-the-art DDD solvers is already not very large, leaving limited scope for further reductions.
We note that for instances requiring more refinement than the existing benchmarks, the benefit of our approaches may be significantly larger.
For example, instances with wider time windows $\ell_k - r_k$ or larger capacities $u_a$ relative to commodity quantities $q_k$ may lead to more consolidations, more conflicts and thus more refinement. 

One promising application of our efficient refinement methods is the development of better heuristics for the relaxation, such that relevant conflicts can be discovered quickly while keeping the model size small.
Then, one might get lucky and only have to solve the relaxation once to obtain an optimal solution.
So far, there is also no satisfactory treatment of cases where dispatch-node graphs contain cycles; especially if commodity time windows are long, eliminating conflicts involving cycles may require many time points, even under optimal refinement.
On the theoretical side, a practical characterisation of an eliminated too-long path would be interesting, though we do not expect this to lead to significantly smaller discretisations.
Similarly, existing approaches are based on preventing conflicts from recurring in \emph{any feasible} solution, but it would be sufficient to prevent them from recurring in \emph{optimal} solutions.
It would also be interesting to extend our methodology to other problems or service network design variants that require different refinement strategies to accommodate extensions such as hub-capacity constraints \cite{he_exact_2022} or holding costs \cite{shu_holding_cost_2024}.

%% file: sections/relaxation-mip.tex
\section{Integer Program for the Relaxed Problem}\label{appendix:IP}
We construct a time-expanded network $D_\calT$ to formulate the integer program for \rsnd with discretisation $\calT$.
Each arc $a \in \A{D}$ receives a timed copy for each interval in $\calT^I_a$, so we obtain the transport arcs
\begin{equation*}
A_\mathrm{T} = \Set{(a,I) : a \in \A{D}, I \in \calT^I_a}.
\end{equation*}
We also identify arcs this way and explicitly specify their head and tail nodes later.
Each node receives a timed copy for each arrival time obtained by dispatching at the start of an interval on an incoming arc.
We also add a timed copy at the beginning of the time horizon.
Formally, we obtain the derived node time discretisation $\calT_v = \Set{t + \tau(a): a \in \Ain{v}, [t,t') \in \calT^I_a} \cup \Set{0}$ and the time nodes
\begin{equation*}
    \V{D_\calT} = \Set{(v,t) \colon v \in \V{D}, t \in \calT_v }.
\end{equation*}
Then, for an arc $b = (a,I) = (uv,[t,t')) \in \A{D_\calT}$ we say that $\head(b) = (v,t+\tau(a))$ and $\tail(b) = (u, \max\Set{t'' \in \calT_u: t'' \leq t'-1})$.
We also specify holding arcs $A_\mathrm{H}$ connecting successive timed copies of each node.
The complete arc set is then $\A{D_\calT} = A_\mathrm{T} \cup A_\mathrm{H}$.
For each commodity we specify the subnetwork $D_{\calT,k}$ available to that commodity by removing arcs $(a,I)$ with $I \cap I_k(a) = \emptyset$.
Finally, we extend the arc-specific parameters to the time-expansion, i.e., for each $b = (a,I) \in A_\mathrm{T}$ we get $c_k(b) = c_k(a), f(b) = f(a),$ and $u(b) = u(a)$. 
Holding arcs incur no cost.
We define each commodity's timed origin and destination as the nodes $\hat{o}_k = (o_k,0)$ and $\hat{d}_k = (d_k,\max \calT_{d_k})$.

We can now state the following integer program (IP) for \rsnd:
\begin{subequations}\label{a:mip}
\begin{alignat}{2}
    \min \quad & \sum_{k \in \calK}\sum_{b \in \A{D_{\calT,k}}} c_k(b) x_{k,b} + \sum_{b \in A_\mathrm{T}} f(b) y_b && \label{a:rp:obj}\\
    \mathrm{s.t.} \quad & \sum_{\mathclap{b \in \Aout{w}}} x_{k,b} - \sum_{\mathclap{b \in \Ain{w}}} x_{k,b}
    = \begin{cases}
        1 & \text{if } w = \hat{o}_k,\\
        -1 & \text{if } w = \hat{d}_k,\\
        0 & \text{otherwise}
    \end{cases}
    & \quad & \forall k \in \calK,\ w \in \V{D_{\calT,k}} \label{a:rp:flowcon}\\
    & \sum_{\mathclap{k \in \calK\colon b \in \A{D_{\calT,k}}}} q_k x_{k,b} \leq u(b) y_b
    && \forall b \in A_\mathrm{T} \label{a:rp:cap}\\
    & x_{k,b} \in \{0,1\} && \forall k \in \calK,\ b \in \A{D_{\calT,k}} \\
    & y_b \in \NN && \forall b \in A_\mathrm{T} \label{a:rp:y}
\end{alignat}
\end{subequations}
For each arc $b=(a,I) \in \A{D_{\calT,k}}\cap A_\mathrm{T}$, the variable $x_{k,b}$ takes the value 1 if and only if $a \in \spathsym_k$ and $h(c(a,k)) = I$.
We assume (without loss of optimality) that all commodities that traverse arc $a$ in the same interval $I$ are consolidated together, and $y_{(a,I)}$ then denotes the number of vehicles necessary for this consolidation.

\begin{remark}
   Our definitions of SND and R-SND require each commodity to follow a path (i.e., with no node repetitions).
   For the relaxation and the theoretical results, we could in fact also allow walks; the results remain valid, but the notation becomes cumbersome.
   We could add a constraint 
   \begin{equation*}
    \sum_{(a,I) \in \A{D_{\calT,k}} \colon a \in \Aout{v}} x_{k,(a,I)}  \leq 1 \quad \forall k \in \calK, v \in \V{D}
   \end{equation*} 
   to only obtain paths in the relaxation, but as previously noted~\cite{helber_arc-based_2026}, this tends to make the relaxation harder to solve without substantially improving the lower bound.
   Therefore, we omit this constraint when solving R-SND. 
\end{remark}

The IP evidently correctly computes the cost of a solution; it remains to show that each of its optimal solutions can be mapped to a solution of R-SND with the same value and vice versa.
First, we note that R-SND allows multiple consolidations using the same interval, while the IP does not. 
But in this case, we could merge the consolidations, preserving feasibility and the objective value (otherwise the solution would not have been optimal). 
After merging any such consolidations, the feasible paths and consolidations are exactly the same.
By construction of $D_{\calT,k}$, a commodity $k$ may select only arc-interval combinations $(a,I)$ fulfilling~\cref{def:representable:nonemptyintervals}.
Two successive timed transport arcs $(a,I)$ and $(a',I')$ can be selected for a commodity in the IP if and only if $\min I + \tau(a) \leq \max I'$ (fulfilling~\cref{def:representable:linking}).
To see this, consider that $(a,I)$ arrives at the intermediate node exactly at time $\min I + \tau(a)$ and $(a',I')$ leaves from $t = \max \Set{t' \in \calT_v : t' \leq \max I'}$. 
These two timed nodes are identical (or linked by holding arcs) if and only if the condition holds.

Next we specify how to solve the relaxed problem when we replace~\cref{def:representable:linking} by the stronger condition where we set $J_k(a,I) = I \cap I_k(a)$ and demand
\begin{equation*}
   \min J_k(a_i,h(c_i)) + \tau(a_i) \leq \max J_k(a_{i+1},h(c_{i+1})) \quad \forall i \in (n).
\end{equation*}
In fact, it suffices to change the commodity-specific time-expanded networks. We can then use exactly the same IP.
First, we specify for each commodity and each node the time discretisation
\begin{equation*}
   \calT_{k,v} =  \{\edt{v}\} \cup \{\min J_k(a,I) + \tau(a) \colon\, a \in \Ain{v}, I \in \calT^I_a, J_k(a,I) \neq \emptyset\}
\end{equation*}
which adds time points for the earliest arrivals for all intervals on incoming arcs, while also considering the commodity's dispatch window $I_k(a)$.
For each commodity, we then construct the following timed node copies:
\begin{equation*}
    \V{D_{\calT,k}} = \{(v,t) \colon\, v \in \V{D}, \edt{v} \leq \ldt{v}, t \in \calT_{k,v}\}.
\end{equation*}
The timed transport and holding arcs are as before, with a slight modification to the nodes connected by each transport arc.
For $b = (a,I)$ with $a = uv$ we set $\head_k(b) = (v, \min J_{k}(a,I) + \tau(a))$ and $\tail_k(b) = (u, \max \Set{t \in \calT_{k,u} :  t \leq \max J_k(a,I)})$.
Finally, we define $\hat{o}_k = (o_k,r_k)$ and $\hat{d}_k = (d_k, \max \calT_{k,d_k})$.
By essentially the same argument as before, the IP on these time-expanded networks solves the stronger R-SND variant.